\documentclass[12pt, oneside, btags, reqno]{amsart}
 \usepackage{graphicx}
\usepackage{amsmath}
\usepackage{amsthm}
\usepackage{amsfonts,amscd,amssymb,mathrsfs,hyperref}
\usepackage[abbrev]{amsrefs}
\usepackage{amscd}   
\usepackage[all]{xy} 
\usepackage{booktabs} 

\usepackage{nccmath} 
\DeclareMathAlphabet{\mathscr}{OT1}{pzc}{m}{it} 
\makeatletter
\g@addto@macro\normalsize{%
  \setlength\abovedisplayskip{4pt}
  \setlength\belowdisplayskip{4pt}
  \setlength\abovedisplayshortskip{4pt}
  \setlength\belowdisplayshortskip{4pt}
}
\makeatother
\newcommand{\lqedhere}{\ensuremath{\text{\qedhere\hspace{-1em}}}}

\usepackage{hyperref}
  \hypersetup{colorlinks=true,citecolor=blue, urlcolor= cyan}

 \AtBeginDocument{%
    \def\MR#1{}
 }

\usepackage{color}

\newtheorem{theorem}{Theorem}[section]

\newtheorem{lemma}[theorem]{Lemma}

\newtheorem{prop}[theorem]{Proposition}

\newtheorem{cor}[theorem]{Corollary}

\theoremstyle{definition}

\newtheorem{examplex}{Example}
\newenvironment{example}
  {\pushQED{\qed}\examplex}
  {\popQED\endexamplex}

\theoremstyle{remark}
\newtheorem{remark}[theorem]{Remark}

\newcommand{\isom}{\cong}

\newcommand{\Res}{\operatorname{Res}}
\newcommand{\Id}{\text{Id}}

\newcommand{\Gr}{\operatorname{Gr}}
\newcommand{\SL}{\operatorname{SL}}
\newcommand{\SO}{\operatorname{SO}}
\newcommand{\GL}{\operatorname{GL}}

\newcommand{\V}{{\mathcal V}}

\newcommand{\PP}{\mathbb{P}}

\newcommand{\CC}{\mathbb{C}}
\newcommand{\ZZ}{\mathbb{Z}}

\newcommand{\Chow}{\operatorname{Chow}}
\newcommand{\Seg}{\operatorname{Seg}}
\newcommand{\Sing}{\operatorname{Sing}}

\def\bw#1{{\textstyle\bigwedge^{\hspace{-.2em}#1}}}
\def\phi{ \varphi }

\def \a{\alpha}

\def \g{\mathfrak{g}}
\def \e{\mathfrak{e}}

\def \sl{\mathfrak{sl}}
\def \so{\mathfrak{so}}

\def \dto{\dashrightarrow}

\DeclareMathOperator{\tr}{Tr}

\begin{document}
\date{\today}

\author{Fr\'ed\'eric Holweck}\email{frederic.holweck@utbm.fr}
\address{Laboratoire Interdisciplinaire Carnot de Bourgogne, ICB/UTBM, UMR 6303 CNRS,
Universit\'e Bourgogne Franche-Comt\'e, 90010 Belfort Cedex, France
}

\author{Luke Oeding}\email{oeding@auburn.edu}
\address{Department of Mathematics and Statistics,
Auburn University,
Auburn, AL, USA
}
\address{Laboratoire Interdisciplinaire Carnot de Bourgogne, ICB/UTBM, UMR 6303 CNRS,
Universit\'e Bourgogne Franche-Comt\'e, 90010 Belfort Cedex, France
}
\address{Institute of Mathematics, Polish Academy of Sciences (IM PAN),
Warsaw, Poland
}

\title{Hyperdeterminants from the $\text{E}_8$ Discriminant}
\begin{abstract}
We find expressions of the polynomials defining the dual varieties of Grassmannians $\Gr(3,9)$ and $\Gr(4,8)$ both in terms of the fundamental invariants and in terms of a generic semi-simple element. We restrict the polynomial defining the dual of the adjoint orbit of $\text{E}_{8}$ and obtain the polynomials of interest as factors. To find an expression of the $\Gr(4,8)$ discriminant in terms of fundamental invariants, which has $15,942$ terms, we perform interpolation with mod-$p$ reductions and rational reconstruction.  From these expressions for the discriminants of $\Gr(3,9)$ and $\Gr(4,8)$ we also obtain expressions for well-known hyperdeterminants of formats $3\times 3\times 3$ and $2\times 2\times 2\times 2$.
\end{abstract}
\maketitle
\section{Introduction}
Cayley's $2\times 2\times 2$ hyperdeterminant is the well-known polynomial
\begin{multline*}
\Delta_{222} = 
x_{000}^{2}x_{111}^{2}
+x_{001}^{2}x_{110}^{2}
+x_{010}^{2}x_{101}^{2}
+x_{100}^{2}x_{011}^{2}
+4(x_{000}x_{011}x_{101}x_{110}+x_{001}x_{010}x_{100}x_{111})\\
-2(x_{000}x_{001}x_{110}x_{111}+x_{000}x_{010}x_{101}x_{111}+x_{000}x_{100}x_{011}x_{111}+\\
x_{001}x_{010}x_{101}x_{110}+x_{001}x_{100}x_{011}x_{110}+x_{010}x_{100}x_{011}x_{101})
.\end{multline*}
It generates the ring of invariants for the  group $\SL(2)^{\times 3}{\ltimes \mathfrak {S}_{3}}$ acting on the tensor space $\CC^{2\times 2 \times 2}$. It is well-studied in Algebraic Geometry. Its vanishing defines the projective dual of the Segre embedding of three copies of the projective line (a toric variety)  \cite{GKZ}, and also coincides with the tangential variety of the same Segre product \cite{LanWey_tan, OedingRaicu, MichalekOedingZwiernik}.  On real tensors it separates real ranks 2 and 3 \cite{ComonOtt_Binary}.
It is the unique relation among the principal minors of a general $3\times 3$ symmetric matrix \cite{HoltzSturmfels}. It is the starting point for many interesting studies in combinatorics \cite{GKZ}. In Computer Vision, the singular locus of its $2\times 2\times 2\times 2$ cousin \cite{WeymanZelevinsky_sing, HSYY_hyperdet} is the variety of quadrifocal tensors for flatlander cameras \cite{Oeding_quad}, and whose defining equations cut out the relations among principal minors of a general $4\times 4$ matrix \cite{LinSturmfels}.
In Quantum Information hyperdeterminants are used as a measure of entanglement \cite{gour2010, SarosiLevay} and can also be investigated to stratify the ambient space of multiqubit systems \cite{miyake_hyperdet, holweck_entanglement, holweck_4qubit2, LevayHolweck}.

Since Sylvester, Schl\"afli, and Cayley in the 19th century, efficient expressions of resultants, discriminants and hyperdeterminants have been key for solving polynomials. General resultants are provably difficult to compute \cite{HillarLim09mosttensor, sparseElimination}. On the other hand, for polynomials with extra structure like those that often come from applications \cite{DAndreaPoisson,HongMinimair, CCDRS_Discriminants, ErimisPan}, sparse resultants can often be computable because of the extra structure they inherit \cite{CannyEmiris, MR1251972, sparseElimination}.

Consider quadratic forms viewed  as square matrices $(a_{ij})$ restricted by linear dependencies causing the matrix to be symmetric. By projecting away from the skew-symmetric part we use roughly half as many parameters $p_{ij} = \frac{1}{2}(a_{ij}+a_{ji})$ with $i\leq j$, hence the matrix determinant applied to a symmetric matrix may be viewed as a sparse resultant. Since determinants are easy to compute (via Gaussian elimination), this is the standard way to compute the discriminant of the quadratic polynomial associated with the symmetric matrix. A central theme of this article is sparse (up to a natural change of coordinates) resultants.

Returning to Cayley's hyperdeterminant, there is a natural projection from $\CC^{2\times 2\times 2}$ to $S^{3}\CC^{2}$ obtained by symmetrizing the coordinates $x_{i,j,k} \mapsto s_{\text{sort}(i,j,k)} = \frac{1}{3!}\sum_{\sigma \in \mathfrak{S}_{3}}  x_{\sigma(i,j,k)}$, which was studied in greater generality in  \cite{oeding_hyperdet}. The restriction of $\Delta_{222}$ (via this projection) is the discriminant $\Delta_3$ of the binary cubic $\sum_{i,j,k \in \{0,1\}} s_{ijk} x_ix_jx_k$, explicitly
\[\Delta_3= s_{000}^{2}s_{111}^{2}-3s_{001}^{2}s_{011}^{2}
-6s_{000}s_{001}s_{011}s_{111}+4(s_{000}s_{011}^{3}+s_{001}^{3}s_{111}).
\]

We can construct $\Delta_{222}$ via restriction of a Grassmann discriminant. Choose a splitting $\CC^{6} = \CC^{2}\oplus \CC^{2}\oplus \CC^{2}$. We obtain a copy of $\CC^{2\times 2\times 2}$ inside $\bw{3}\CC^{6}$ by decomposing $\bw{3}(\CC^{2}\oplus \CC^{2}\oplus \CC^{2})$ as an $\SL(2)^{\times 3}$-module. The projection is given on Pl\"ucker coordinates by $p_{i,j,k} \mapsto x_{a,b,c}$, where
$a =\begin{cases} 0, & \text{if } i = 1 \\ 1, &\text{if } i = 4  \end{cases}$,
$b =\begin{cases} 0, & \text{if } j = 2 \\ 1, &\text{if } j = 5  \end{cases}$,
$c =\begin{cases} 0, & \text{if } k = 3 \\ 1, &\text{if } k = 6  \end{cases}$, and $p_{i,j,k} \mapsto 0$ if $\{i,j,k\}$ contains more than one element from any of the sets $\{1,4\}$, $\{2,5\}$, $\{3,6\}$.
We notice that $\Delta_{222}$ is the restriction of the defining polynomial of the dual of the Grassmannian $\Gr(3,6)$: $\Delta_{\Gr(3,6)}=$
\[\begin{smallmatrix}
{p_{123}}^{2}{p_{456}}^{2}+{p_{124}}^{2}{p_{356}}^{2}+{p_{125}}^{2}{p_{346}}^{2}+{p_{345}}^{2}{p_{126}}^{2}+{p_{134}}^{2}{p_{256}}^{2}+{p_{135}}^{2}{p_{246}}^{2}+{p_{245}}^{2}{p_{136}}^{2}+{p_{145}}^{2}{p_{236}}^{2}+{p_{235}}^{2}{p_{146}}^{2}+{p_{234}}^{2}{p_{156}}^{2}
\\
+4 (p_{123}p_{145}p_{246}p_{356}-p_{123}p_{145}p_{346}p_{256}-p_{123}p_{245}p_{146}p_{356}+p_{123}p_{345}p_{146}p_{256}+p_{123}p_{245}p_{346}p_{156}-p_{123}p_{345}p_{246}p_{156}\\
+p_{124}p_{135}p_{236}p_{456}+p_{124}p_{135}p_{346}p_{256}-p_{124}p_{235}p_{136}p_{456}-p_{124}p_{345}p_{136}p_{256}-p_{124}p_{235}p_{346}p_{156}+p_{124}p_{345}p_{236}p_{156}\\
-p_{134}p_{125}p_{236}p_{456}+p_{134}p_{125}p_{246}p_{356}+p_{234}p_{125}p_{136}p_{456}+p_{125}p_{345}p_{136}p_{246}-p_{234}p_{125}p_{146}p_{356}-p_{125}p_{345}p_{236}p_{146}\\
+p_{134}p_{235}p_{126}p_{456}-p_{134}p_{245}p_{126}p_{356}-p_{234}p_{135}p_{126}p_{456}+p_{135}p_{245}p_{126}p_{346}+p_{234}p_{145}p_{126}p_{356}-p_{235}p_{145}p_{126}p_{346}\\
+p_{134}p_{235}p_{246}p_{156}-p_{134}p_{245}p_{236}p_{156}+p_{234}p_{135}p_{146}p_{256}+p_{135}p_{245}p_{236}p_{146}-p_{234}p_{145}p_{136}p_{256}+p_{235}p_{145}p_{136}p_{246})\\
+2 (-p_{123}p_{124}p_{356}p_{456}+p_{123}p_{125}p_{346}p_{456}-p_{123}p_{345}p_{126}p_{456}+p_{123}p_{134}p_{256}p_{456}-p_{123}p_{135}p_{246}p_{456}+p_{123}p_{245}p_{136}p_{456}\\
-p_{123}p_{145}p_{236}p_{456}+p_{123}p_{235}p_{146}p_{456}-p_{123}p_{234}p_{156}p_{456}-p_{124}p_{125}p_{346}p_{356}+p_{124}p_{345}p_{126}p_{356}-p_{124}p_{134}p_{256}p_{356}\\
-p_{124}p_{135}p_{246}p_{356}+p_{124}p_{245}p_{136}p_{356}-p_{124}p_{145}p_{236}p_{356}+p_{124}p_{235}p_{146}p_{356}+p_{124}p_{234}p_{156}p_{356}-p_{125}p_{345}p_{126}p_{346}\\
-p_{134}p_{125}p_{346}p_{256}-p_{125}p_{135}p_{246}p_{346}-p_{125}p_{245}p_{136}p_{346}+p_{125}p_{145}p_{236}p_{346}+p_{125}p_{235}p_{146}p_{346}+p_{234}p_{125}p_{346}p_{156}\\
+p_{134}p_{345}p_{126}p_{256}-p_{135}p_{345}p_{126}p_{246}-p_{245}p_{345}p_{126}p_{136}+p_{145}p_{345}p_{126}p_{236}+p_{235}p_{345}p_{126}p_{146}-p_{234}p_{345}p_{126}p_{156}\\
-p_{134}p_{135}p_{246}p_{256}+p_{134}p_{245}p_{136}p_{256}+p_{134}p_{145}p_{236}p_{256}-p_{134}p_{235}p_{146}p_{256}-p_{134}p_{234}p_{156}p_{256}-p_{135}p_{245}p_{136}p_{246}\\
-p_{135}p_{145}p_{236}p_{246}-p_{135}p_{235}p_{146}p_{246}-p_{234}p_{135}p_{246}p_{156}-p_{145}p_{245}p_{136}p_{236}-p_{235}p_{245}p_{136}p_{146}+p_{234}p_{245}p_{136}p_{156}\\
-p_{235}p_{145}p_{236}p_{146}+p_{234}p_{145}p_{236}p_{156}-p_{234}p_{235}p_{146}p_{156}).
\end{smallmatrix}\]

Our aim is to explain these coincidences geometrically and provide several generalizations. In particular, we will find the defining polynomials for the duals of $\Gr(3,9)$ and $\Gr(4,8)$ by restricting discriminants of the adjoint orbits of $\text{E}_{8}$ and $\text{E}_{7}$ respectively, and subsequently use them to recover the $3\times 3\times 3$ and $2\times 2\times 2\times 2$  hyperdeterminants (again by restricting the discriminants of the Grassmannians). These examples fit nicely into the story of Vinberg's $\theta$-representations, (see \cite{Gruson2013} for connections with moduli of abelian varieties and free resolutions) and this rich theory helps make our computations manageable.

In Section~\ref{sec:mtangent} we give Theorem~\ref{thm:tan} which generalizes a lemma from the famous textbook \cite{GKZ} and allows us to study projections (restrictions) of polynomials defining dual varieties.
 In Section~\ref{sec:lie} we restrict the discriminants $\Delta_G$ of  adjoint varieties for $G=\text{E}_8,\text{E}_7,\text{E}_6$ and $\SO(8)$.  We obtain divisibility relations between the restrictions of discriminants of certain Grassmannians and of discriminants of adjoint orbits (see Figure~\ref{fig:table}).
In Section~\ref{sec:semi-simple} we describe a Jordan-Chevalley decomposition for $\mathfrak{e}_{8}$ that allows us to interpret elements of $\bw{3}\CC^{9}$ and $\bw{4}\CC^{8}$ as being semi-simple, which in turn allows us to evaluate the restriction of the $\text{E}_{8}$ discriminant to generic semi-simple elements of $\bw{3}\CC^{9}$ and $\bw{4}\CC^{8}$ respectively. 
In Section~\ref{sec:fundamental} we construct expressions for the $\Gr(3,9)$ and $\Gr(4,8)$ discriminants as polynomials in fundamental invariants using linear interpolation, reductions mod $p$, and rational reconstruction. 
In Section~\ref{sec:gr39gr48} we give methods to evaluate the $\Gr(3,9)$ and $\Gr(4,8)$ discriminants. In particular, we make use of Katanova's explicit expressions for the fundamental invariants as traces of powers of special matrices \cite{Katanova}. We also describe their respective restrictions to tensors of format $3\times 3\times 3$ and $2\times 2\times 2\times 2$.

\section{Decomposing restrictions of dual varieties}\label{sec:mtangent}
Let $V$ denote an $n$-dimensional vector space over $\CC$ with basis $e_{1},\ldots,e_{n}$. 
A splitting $V = A\oplus B$, which we always assume to be non-trivial, induces a splitting on the dual, $V^{*} = A^{*}\oplus B^{*}$, where $A^{*}$ are linear functionals on $A$  extended to linear functionals on $V$ by setting their value to $0$ on $B$.
The homogeneous polynomials of degree $d$  on a basis of $V$ are in one-to-one correspondence with the space of symmetric $d$-forms on $V$, denoted $S^{d}  V^{*}$. Since $S^{d}V^{*} = S^{d}A^{*}\oplus \bigoplus _{0\leq i\leq d}S^{d-i}A^{*}\otimes S^{i}B^{*}$ we have linear projections $S^{d} V^{*} \to S^{d}A^{*}$ for every $d\geq 0$ induced by the projection $V^{*} \to A^{*}$. 
This also induces rational maps on projective spaces $\PP V \dto \PP A$ and $\PP V^* \dto \PP A^*$. The indeterminacy locus of a map induced from a linear map of vector spaces is the projectivization of the kernel. By convention, for any subvariety $X \subset \PP V$, the image of $X$ under a rational map is the closure of the image of the (regular) map on a relatively open subset of $X$ where the map is defined. We will use $\pi_B$ to denote the maps induced from the projection $V\to A$ (affine / projective).

We will also project defining polynomials $\Delta$ (discriminants) in $\PP V $. Such  elements of $S^d(A\oplus B)$ and can be projected to $S^d(A)$ by restricting $\Delta\in S^d V$ to the subspace $A^*$. To not confuse the two types of projection we will use $\Delta_{|A^*}=\Res (\Delta,A^*)$ to denote the projection/restriction of the discriminant $\Delta$ to $A^*$. Note there is a unique decomposition $\Delta=\Res (\Delta,A^*)+g$ with $g\in (S^dA^*)^\perp$, where $\perp$ denotes the annihilator, so that ${A^{*}}^{\perp} = B$.

\begin{remark}Given $f\in S^d(A^*\oplus B^*)$ we have $\V( \Res(f,  A^*), B^* )= \V(f)\cap \PP A$, that is to say that the zero-set of the restriction of $f$ to $A$ (still viewed as a polynomial on $V$) together with a set of linear forms making a basis of $B^{*}$ is precisely the intersection $\V(f) \cap \PP A$. 
\end{remark}

For a projective variety  $X\subset \PP V$, let $\widehat T_x X \subset V$ denote the tangent space at $x\in \widehat X \subset V$, where $\widehat{\cdot}$ indicates the cone in $V$. The \emph{projective dual variety} of $X$, denoted $X^{\vee}\subset \PP V^{*}$, is
\[
X^{\vee}:= \overline{ \left \{ [H] \in \PP V^{*} \mid  \exists [x] \in X_{\text{sm}}\;\; s.t.\;\; H\supset \widehat {T}_{x}X \right\} },
\]
the Zariski closure of hyperplanes tangent to $X$ at smooth points, \cite{GKZ}.
If $X$ is irreducible, then (over $\CC$) so is $X^{\vee}$, and when $X^{\vee}$ is a hypersurface, we often denote by $\Delta_{X}$ its defining polynomial, $\V(\Delta_{X}) = X^{\vee}$, and refer to it as a \emph{hyperdeterminant} or a \emph{discriminant}.

Theorem~\ref{thm:tan} below gives a sufficient condition to guarantee that the restriction of the hyperdeterminant is divisible by the equation of the dual of another variety. These questions were considered for hyperdeterminants and symmetric tensors in \cite{oeding_hyperdet}, and for hyperdeterminants and other Schur functors in \cite{tocino}. Similar considerations were also made by Pedersen and Sturmfels \cite[Section~4]{pedersen1993product} and \cite[Theorem~5.1]{pedersen1993product}, and are key to Poisson formulas \cite{DAndreaPoisson}.

\begin{theorem}\label{thm:tan} Consider a non-trivial vector space splitting $V=A\oplus B$.
Let $X\subset \PP V$ and $Y\subset \PP A$ be projective varieties. Let $\pi_B$ denote rational map $\PP V \dto \PP A$ induced from the projection $V \to A$. 
If for each smooth point $[y] \in Y$ there is a smooth point $[x] \in X$ such that $\pi_{B}(\widehat T_{x} X) \subset  \widehat T_{y} Y$ (\emph{the tangency condition}), then
\[
Y^{\vee} \subseteq X^{\vee}\cap \PP A^{*}.
\]
Moreover if $X^{\vee}$ and $Y^{\vee}$ are hypersurfaces defined respectively by polynomials $\Delta_{X}$ and $\Delta_{Y}$ and, 
for every general point $[h]\in Y^{\vee}$,  $H = \mathcal{V}(h)$, viewed as a hyperplane in $\PP V$, is a point of multiplicity $m$ of $X^{\vee}$  then
\[
\Delta_{Y}^{m}\mid \Res(\Delta_{X},A^*).
\]
\end{theorem}
\begin{proof}
We identify $V^{*} = A^{*}\oplus B^{*}$. Suppose $h \in A^{*} \subset V^{*}$ and  $\mathcal{V}(h) \subset A$ is a hyperplane that is general among those hyperplanes tangent to $Y$ at a smooth point $[y]$. Then the tangency condition guarantees that there is a smooth point $[x]\in X$ such that every vector $t \in \widehat T_{x}X$ can be written uniquely as $ t = a+ b$ with $a \in \widehat T_{y}Y$ and $b \in B$. Since $h$ is a linear form we have $h(t) = h(a)+h(b)$,  and $h(a) = 0$ because $a \in \widehat T_{y}Y \subset A$, and $h(b) = 0$ for all $b \in B$ since $h\in A^*$. Therefore, $h$ vanishes at every point of $\widehat T_{x}X$,  and $H =\V(h)$ is a hyperplane tangent to $X$. This concludes the first part.

Now for the second part.  Let $\Sing_{m-1}(X^{\vee})$ denote the closure of the points of multiplicity $m$ in $X^{\vee}$.

\textbf{Step 1}: If $[h] \in Y^{\vee}$ and $[h] \in \Sing_{m-1}(X^{\vee})$ then $[h] \in \Sing_{m-1}(X^{\vee} \cap \PP A^{*})$.

Write $\Delta_{X} = \Res (\Delta_{X},A^*) + \mathcal{O}(B)$, with $\mathcal{O}(B)$ denoting those terms with at least one variable in $B$.
Evaluate at $h$: $\Delta_{X}(h) = \Res (\Delta_{X},A^*)(h) + \mathcal{O}(B)(h) = 0$ since by the first part we get  $\Res (\Delta_{X},A^*)(h) = 0$  and $\mathcal{O}(B)(h) = 0$ because $h$ doesn't use any $B$-variables.
Taking derivatives,
\[\frac{\partial^{k} \Delta_{X}}{\partial A^{k}} =
 \frac{\partial^{k} \Res (\Delta_{X},A^*)}{\partial A^{k}} + \frac{\partial^{k} \mathcal{O}(B)}{\partial A^{k}}
 ,\] 
 where $\partial A^{k}$ stands for any $k$-order partial derivative with respect to variables in $A$. Now evaluate at $h$ for $k \leq m-1$ to obtain
\[0 =
 \frac{\partial^{k} \Res (\Delta_{X},A^*)}{\partial A^{k}}(h) +0
. \]
So $[h]$ is of multiplicity at least $m$ on $X^{\vee} \cap \PP A^{*}$.

\textbf{Step 2}: If for all $[h] \in Y^{\vee}$ we have that $[h] \in \Sing_{m-1}(X^{\vee} \cap \PP A^{*})$ then $\Delta_{Y}^{m} \mid \Res (\Delta_{X},A^*).$

Without loss of generality, suppose $[h]$ is a smooth point of $Y^{\vee}$ throughout. By induction on $m$, with base case being the first case of the theorem, we assume that $[h]\in Y^{\vee} \cap \Sing_{m-2}(X^{\vee})$ implies that $\Res (\Delta_{X},A^*) = \Delta_{Y}^{m-1}\cdot F$. We want to show that $[h]\in Y^{\vee} \cap \Sing_{m-1}(X^{\vee})$ implies that $\Delta_{Y}$ divides $F$.
Now we compute
\[
\frac{\partial ^{m-1} \Res (\Delta_{X},A^*)}{\partial A^{m-1}} = (m-1)!\left (\prod_{a_j \in A^{m-1}} \frac{\partial \Delta_Y}{\partial a_j}\right) F + \Delta_{Y}\cdot G
.\]
By smoothness of $Y^{\vee}$ at $[h]$ there exists a collection  $A^{(m-1)}$ of $m-1$ (possibly not distinct) variables of $A$ such that
\[\left(\prod_{a_j \in A^{(m-1)}} \frac{\partial \Delta_Y}{\partial a_j}\right)(h)=\lambda\neq 0.
\]
Therefore, one obtains:
\[
0 = (m-1)!\lambda F(h) + 0
,\]
showing that $\Delta_{Y} \mid F$ because $h$ was general and $\Delta_{Y}$ is irreducible.
\end{proof}
\begin{remark}
In the case that $Y$ satisfies the tangency condition and $X^\vee$ and $Y^\vee$ are hypersurfaces, if there exist two distinct smooth points $[x_1]$, $[x_2]$ in $X$ such that $\pi_{B}(\widehat T_{x_{i}} X) \subset  \widehat T_{y} Y$, then $m\geq 2$.
This is because for a general $[h]\in Y$, the hyperplane $H=\mathcal{V}(h)\subset \PP V$ has (at least) two points of tangency $[x_1]$ and $[x_2]\in X$ one knows that $\text{mult}_{[h]} X^\vee \geq 2$.
\end{remark}

\begin{remark}\label{rem:paru}
A result of Parusi\'nski \cite{Parusinski} implies that the hypothesis $[h]$ is a singular point of $X^{\vee}$ of multiplicity $m$ is equivalent to the statement that the Milnor number of the singular section $X\cap \mathcal{V}(h)$ is $m$. In  particular if $\mathcal{V}(h)$ has $k$ points of tangency defining each a Morse singularity, i.e., the quadratic part of the singularity is of full rank, one directly obtains $m=k$. More generally if the points of tangency of $\mathcal{V}(h)$ are isolated, then $m=\text{mult}_{[h]} X^\vee$ is  greater than the number of points of tangency.
\end{remark}
\begin{remark}
The inclusion of the \emph{tangency condition} corrects a mistake in \cite[Prop.~4.5]{oeding_hyperdet} that neglected this condition.
In \cite{oeding_hyperdet} two proofs of a component of the main result were given, one of which was independent of \cite[Prop.~4.5]{oeding_hyperdet}. So, one could safely delete \cite[Prop.~4.5]{oeding_hyperdet} and  Proof~1 of \cite[Lem.~5.1]{oeding_hyperdet} and the results of the article would not change.
On the other hand, the tangency condition is satisfied by the varieties considered in \cite[Lem.~5.1]{oeding_hyperdet}, so Theorem~\ref{thm:tan} could be used in place of \cite[Prop.~4.5]{oeding_hyperdet} in Proof~1 of \cite[Lem.~5.1]{oeding_hyperdet}.
\end{remark}

We recall a foundational result of Gelfand-Kapranov-Zelevinsky: 
\begin{prop}[{\cite[Prop.~I.4.1 and Cor.I.4.2]{GKZ}}]\label{prop:gkz}
Consider a variety $X \subset \PP V = \PP( A\oplus B)$.
\begin{enumerate}
\item \label{gkz:proj} Suppose $\widehat X\cap B = 0$ and $\dim \widehat X < \dim A$.
\item\label{gkz:projiso} Suppose further that $\pi_{B}\colon X \dto \pi_{B}(X)$ is an isomorphism of algebraic varieties. \end{enumerate}
Then  for $Y = \pi_{B}(X) $ we have
\[
Y^{\vee} \subseteq X^{\vee}\cap \PP A^{*} \quad\text{in case \eqref{gkz:proj}} \quad and \quad Y^{\vee} = X^{\vee}\cap \PP A^{*} \quad \text{in case \eqref{gkz:projiso}.}
\]
If $X^{\vee}$ and $Y^{\vee}$ are hypersurfaces defined respectively by polynomials $\Delta_{X}$ and $\Delta_{Y}$, then
\[
\Delta_{Y}\mid \Res (\Delta_{X},A^*)  \quad\text{in case \eqref{gkz:proj}  \quad and \quad} \left[\Delta_Y \right]  = \left[ \Res (\Delta_X,A^*) \right] \quad \text{in case \eqref{gkz:projiso}}.
\]
\end{prop}

Theorem~\ref{thm:tan} allows us to expand the scope slightly. The requirement $X \cap \PP B = \emptyset$ seems to be in place to make sure that the projection is well defined on projective space -- this hypothesis is rendered unnecessary by considering the image of the rational map. In fact, we will see situations where it is desirable to take this more flexible viewpoint.

\begin{cor}\label{cor:proj}
The image of the rational map $Y_\pi:=\pi_B(X)$ satisfies the tangency condition, and hence $ Y_\pi^\vee \subset X^\vee \cap \PP A^*$.
\end{cor}

\begin{proof}
It suffices to verify the tangency condition on an open subset and apply Theorem~\ref{thm:tan}. 
Let $Y' = \{ \pi_B(x) \in \PP A \mid x \in X_{\text{sm}}-\PP B\} $ (without taking closure).  We claim points in the interior of $Y'$ satisfy $\pi_B(\widehat T_x X) = \widehat T_{\pi_B(x)}Y'$. 
A vector is in the tangent space $\widehat T_{\pi_B(x)}Y'$ if and only if it is of the form $a'(0)$ for some smooth curve $[a(t)]\subset Y'$ with $a(0) = \pi_B(x)$. 
A vector is in $\widehat T_x X$ if and only if  it is of the form $\gamma'(0) = a'(0) + b'(0)$ for some smooth curve $[\gamma(t)] \subset X$ with $\gamma(0) = x$ and $\gamma(t) = a(t) + b(t)$ with $a(t) \in A$ and $b(t) \in B$ for all $t$. Since $\pi_B(\gamma'(0)) = a'(0)$ we have ``$\subset$''. Since every point in an open neighborhood of $[y] \in Y'$ has a preimage in $X_{\text{sm}}-\PP B$, every (local) curve $[a(t)]$ in $Y'$ through $[y]\in Y'$ has a preimage curve $[\gamma(t)]$ in $X$, and such that $\pi_B(\gamma(t)) = a(t)$. Without loss of generality we can pass to a subset of $X\cap \gamma(t)$ to insist that  $\gamma$ is smooth and $\pi_B$ is one-to-one on $X\cap \gamma(t)$. Thus we can construct a curve $b(t)$ by considering $b(t):= \gamma(t) - a(t)$. Thus every tangent vector in $\widehat T_yY'$ has a preimage in $\widehat T_xX$, so we obtain ``$\supset$.''
\end{proof}
\begin{remark}
Assume $X^\vee$ is a hypersurface. To understand the factoring of its restriction we want to know which of the subvarieties showing up in the geometric decomposition are not dual-defective (their dual varieties are hypersurfaces). The multiplicities occurring in the factorization can be understood by Theorem~\ref{thm:tan}, or in cases when few factors are present a degree argument may suffice. As noted in \cite[p.33]{GKZ}, points of $Y= \pi_B(X)$ on the boundary of $Y' = \pi_B(X- \PP B)$ may also contribute to the restriction of the dual (see Example~\ref{ex:curve}).
\end{remark}

Even for special subvarieties of the projection, like intersections $Y_\cap = X\cap \PP A $, the tangency condition is not automatic, as shown in Example~\ref{ex:traceless} where $Y_\cap$ \emph{does not} satisfy the tangency condition and does not contribute to the restriction of the dual. Nevertheless we can try to find subvarieties of $Y = \pi_B(X)$  whose dual might contribute to the restriction of the dual of $X$ by considering the locus where the projection drops rank, what we call the \emph{$\pi_B$-rank loci}, which we define below. 

Let $\underline V$ denote the tautological vector bundle on a variety $X\subset \PP V$. 
Suppose $V = A\oplus B$ and let $\pi_{B}$ denote the projection $\underline V \to \underline A$, the fiber-wise projection $V \to A$. Similarly define $\pi_B$ on the cone over the total space $X \times \PP \underline V$:
\begin{eqnarray*}
\pi_{B} \colon \widehat X\times  \underline V &\to&  A \times  \underline A\\
(x,v) &\mapsto& (\pi_{B}(x), \pi_{B}(v)).
\end{eqnarray*}
The map $\pi_B$ also restricts to the tangent sheaf $\mathcal{T}X$. 
By definition, the map $\pi_B \colon  \widehat T_x X \to A$ has rank at most $\dim \widehat X$ for all $[x] \in X_{\text{sm}}$. 
For each $k$ with  $0\leq k \leq\dim \widehat X$ define the points
\[ \tilde Y_k = \overline{ \{ (\pi_B(x), \pi_B(\widehat T_x X) ) \mid \dim \pi_B(\widehat T_x X) = k, \text{ and } [x] \in X_{\text{sm}}\}}.\]
Define the rank locus $Y_k$ as the projection onto the first component.  We can then ask if these rank loci satisfy the tangency condition.

\subsection{Examples}
Now we illustrate our geometric technique in several cases, including when the center of projection intersects the variety $X$,  classical determinants and Pfaffians, and the example in the introduction.
We often consider the  Grassmannian $\Gr(k,V) \subset \PP \bw{k} V$, the Segre variety $\Seg(\PP V^{\times d}) \subset \PP V^{\otimes d}$, and the Veronese variety $\nu_{d}\PP V \subset \PP S^{d}V$, which are homogeneous $G$-varieties for $G$ respectively $\SL(V)$, $\SL(V)^{\times d}\ltimes \mathfrak{S}_{d}$, or $\SL(V)$.  For a partition $\lambda$ of an integer $d>0$ with $t$ parts the Chow variety, $\Chow_{\lambda}\PP V$,  is the projection to symmetric tensors of the Segre-Veronese variety $\Seg_{\lambda}(\PP V)^{|\lambda|}$. Specifically, the general point $[v_{1}^{\otimes \lambda 1}\otimes \cdots \otimes v_{t}^{\otimes \lambda_{t}}] \in \Seg_{\lambda}(\PP V)^{|\lambda|}$  projects to the general point  $[v_{1}^{ \lambda 1}\circ \cdots \circ v_{t}^{\lambda_{t}}] \in \Chow_{\lambda}\PP V$.

First we consider a nice non-example, that explains a case when the restriction is zero.
 \begin{prop}\label{prop:center}
Consider a variety $X \subset \PP V = \PP(A \oplus B)$.  If for all $h \in A^{*}\subset V^{*}$ there is an $[x]\in X_{\text{sm}}$ such that $h(\widehat T_{x} X) = 0$ then $\Res(\Delta_{X},A^*) \equiv 0$.
 \end{prop}
 \begin{proof}
Our assumptions imply that  $\PP A^{*} \subset  X^{\vee} \cap \PP A^{*}$. So $X^{\vee} \cap \PP A^{*} = \PP A^{*}$.
 \end{proof}
 Tocino \cite{tocino} studied the case when $X$ is the Segre variety and noted that all but two Schur functors (projections) $S_{\lambda}\colon V^{\otimes d}\to V^{\otimes d}$  satisfy the hypotheses of Proposition~\ref{prop:center}, and thus most hyperdeterminants of format $n^{\times d}$ restrict to zero for special symmetry types of tensors.

\begin{example}\label{ex:traceless} The following illustrates a case when the tangency condition fails.
Let $X = \Seg(\PP V^{*}\times \PP V)$, the projective variety of rank-1 linear transformations $V\to V$, with $n=\dim V$, and $n\geq 3$. 
As $\SL(V)$-modules $V^{*}\otimes V = \Gamma^{n-1,1}\oplus \Gamma^{0}$, where $\Gamma^{n-1,1} \cong \sl(V)$ is the space of traceless linear transformations $V\to V$ and $\Gamma^{0}\cong \CC$ is the span of the identity transformation.
Consider the projection (which does not generally preserve matrix rank)
\begin{eqnarray*}
\pi_{\Gamma^0}\colon V^{*}\otimes V & \to&  \Gamma^{n-1,1} \\
A&\mapsto& A-\frac{\tr(A)}{n}\Id.
\end{eqnarray*} 
Denote the image by $Y_\pi :=\pi_{\Gamma^0}(X)$, which consists of points of the form $[\a\otimes b - \frac{\a(b)}{n}\Id_{V}]$, with $\Id_{V}$ the identity in $V^{*}\otimes V$, and $n=\dim V$. 
A straightforward calculation shows that 
\[\widehat T_{\pi_{\Gamma^0}( \a\otimes b)} \pi_{\Gamma^0} X = \{ \a'\otimes b + \a\otimes b' - \frac{\a'(b)+\a (b')}{n}\Id_{V} \mid \;\text{ with } \a' \in V^* \;\text{ and } b'\in V\}.
\]
The Segre variety has tangent space $\widehat T_{\a \otimes b} \Seg(\PP V^* \times \PP V) =  V^* \otimes b + \a \otimes V$, whence one sees that the projection of the tangent space is equal to the tangent space of the projection. This is an example of Corollary~\ref{cor:proj}, which implies that $Y_\pi^\vee \subset X^\vee \cap \PP \Gamma^{n-1,1}$.  We note also that here the hypotheses of Proposition~\ref{prop:gkz}\eqref{gkz:proj} are also satisfied:  $\Seg(\PP V^* \times \PP V )\cap \PP \Gamma^0 = \emptyset$ because the rank of the matrices in the Segre are all 1, which is not the rank of non-zero scalar multiple of the identity for $n> 1$, and $\dim \widehat \Seg(\PP V^* \times \PP V) = 2n-1 < n^2-1 = \dim \Gamma^{n-1,1}$ for $n>2$.

Let us consider the $\pi_{\Gamma^0}$-rank loci.
The map $\pi_{\Gamma^0} \colon \widehat T_{\a\otimes b} \Seg(\PP V^* \times \PP V) \to \Gamma^{n-1,1}$ drops rank precisely for points $\a \otimes b$ with $\a (b) = 0$, which coincides with $Y_{\cap} = \{[\a\otimes b] \mid \a(b) =0\}$, the rank-1 traceless linear transformations $V\to V$. Note that $Y_\cap$ has codimension 1 in $Y_\pi$. 
The dual $Y_\cap^{\vee}$ is a degree $n(n-1)$ hypersurface (it is the dual of an adjoint orbit, cut out by the discriminant of the characteristic polynomial), whereas $X^{\vee}$ is a degree $n$ (determinantal) hypersurface. Since the restriction of $\Delta_{X}$ cannot have degree at least $n(n-1)$ if $n\geq 3$, this must be a situation where Theorem~\ref{thm:tan} does not apply.  
Indeed, the respective spaces  $\widehat T_{e^{1}\otimes e_{n}}Y_\cap$ and $\pi_{\Gamma^0}(\widehat T_{e^{1}\otimes e_{n}}X)$  comprise matrices of the forms
\[
\left(\begin{smallmatrix}
r & *&\dots &* &* \\
0 & 0 & \dots &0 &*\\
\vdots & &\ddots &0 & * \\
0 & 0 & \dots &0&-r\\
\end{smallmatrix}\right)
,
\quad\text{and respectively } \left(\begin{smallmatrix}
s & *&\dots &* \\
0 & 0 & \dots &*\\
\vdots & &\ddots & * \\
0 & 0 & \dots &t\\
\end{smallmatrix}\right)
 - \frac{s+t}{n} \Id_{V}.
\]
In the first case elements have rank $\leq 2$, but in the second the matrices have ranks $0,1,2,n-1$ and $n$. These rank conditions are independent of our choices of coordinates.  So, the hypothesis $\exists [x] \in X =\Seg (\PP V^{*}\times \PP V)$ such that $\pi_{\Gamma^0}(\widehat T_{x}X) \subset \widehat T_{y}Y_\cap$ must fail since the projection $\pi_{\Gamma^0}(\widehat T_{x}X)$ contains elements of rank $n\geq3$. 
 Hence Theorem~\ref{thm:tan} only allows us to conclude $Y_{\pi}^\vee \subset X^\vee \cap \PP (\Gamma^{n-1,1})^* $.    
Now, checking that $Y_{\pi}^\vee$ is a hypersurface we find $\Delta_{Y_{\pi}}\mid \Res(\Delta_{X},\Gamma^{n-1,1})$. By computing degrees one can conclude that this is the only divisibility. 
One can also check that the restriction  $\Res(\Delta_X,\Gamma^{n-1,1}) = \det(A - \frac{tr (A)}{n} Id)$ is irreducible and hence $[\Delta_{Y_\pi}] = [ \Res(\Delta_{X},\Gamma^{n-1,1})]$.
\end{example}
\begin{example}\label{ex:mat}
Now we consider the well-known restrictions of the determinant to symmetric and skew-symmetric matrices.
Consider $V \otimes V$ with the usual $\SL(V)\times \SL(V)$ action, thought of as a space of square matrices.
The decomposition $V^{\otimes 2} = S ^{2}V \oplus \bw{2} V$ as $\SL(V)$-modules gives projections that we denote $\pi_{\circ}$ and $\pi_{\wedge}$ respectively. On indecomposable elements  we have $\pi_\circ (a\otimes b) = a\circ b := \frac{1}{2}(a\otimes b + b\otimes a)$, and respectively $\pi_\wedge (a\otimes b) = a\wedge b := \frac{1}{2}(a\otimes b - b\otimes a)$. In bases, for a square matrix $A$ we have $\pi_\circ = \frac{1}{2}(A+A^\top)$ and $\pi_\wedge = \frac{1}{2}(A-A^\top)$, where $\cdot^{\top}$ denotes the matrix transpose. 
 Let $X = \Seg(\PP V \times \PP V) \subset \PP V^{\otimes 2} $ denote  the projective variety of rank-1 square matrices.
The tangent sheaf over $X$ is $TX =\PP \{a \otimes V + V \otimes b \to a\otimes b \in \widehat X \mid a,b\in V\}$.  
We will compute the projection rank loci for the respective projections of $X$. 

The hypotheses of Proposition~\ref{prop:gkz}\eqref{gkz:proj} apply in the symmetric case since there are no rank-1 skew-symmetric matrices, i.e. $\widehat X\cap \bw{2}V = 0$. However, in the skew-symmetric case we cannot apply Proposition~\ref{prop:gkz}\eqref{gkz:proj} since $\widehat X\cap  S^2 V = \widehat \nu_2 \PP V \neq 0$. So, we show how to use Theorem~\ref{thm:tan} to understand both restrictions.

\textbf{Symmetric case}: Consider $\pi_\circ X = \Chow_{1,1} \PP V$, the Chow variety of completely reducible quadrics. The projection of the tangent cone is $\pi_\circ \widehat T_{[a\otimes b]} X = \pi_{\circ}(a \otimes V + V \otimes b) = a\circ V + V\circ b$. For $a$, $b$ both nonzero and not colinear, the space $a\circ V + V\circ b$ is the tangent cone  at the point $[a\circ b]$ of $\Chow_{1,1} \PP V \subset \PP S^2 V$. Moreover such $[a\circ b]$ with $a\wedge b \neq 0$ form the open subset of smooth points of the Chow variety. So, the Chow variety satisfies the tangency condition, an example of Corollary~\ref{cor:proj}.

The projection $\pi_\circ$ drops rank precisely on the closed subset $\{[a\circ b] \mid a\wedge b = 0 \}$, which coincides with the Veronese variety $ \nu_2(\PP V)$. Hence, the $\pi_\circ$-rank loci consist precisely of the Chow and Veronese varieties.
The Veronese variety satisfies the tangency condition: every tangent space of the Veronese is of the form  $\widehat T_{a^{\circ 2}}\nu_2 \PP V = a\circ V$, the point $[a^{\circ 2}]$ is a smooth point of the Veronese, and the tangent space  $\widehat T_{a\otimes a} X$ projects to it.
The only way for $\widehat T_{p\otimes q}X$ to project into a tangent space of $\nu_2\PP V$ is for $[p\otimes q]=[a\otimes a]$ for some $a\in V$. To see this, note that the symmetric matrix $p\circ q$ has image $\textrm{span}\{p,q\}$, which is 1-dimensional if and only if $p$ and $q$ are proportional. Hence the multiplicity of $(\nu_2 \PP V)^\vee$ as a component of the restriction of the determinant to symmetric matrices is $m_1=1$.

Note that $(\Chow_{1,1}\PP V)^{\vee}$ is not a hypersurface \cite[Thm.~1.3]{oeding_hyperdet}. Thus we can only determine that the discriminant of a quadric (the equation of the dual of the Veronese $\nu_{2}(\PP V)$) divides the determinant of a symmetric matrix (up to scale). By comparing degrees we conclude
\[
[\Res(\Delta_{\Seg(\PP V\times \PP V)},S^2 V^*)] = [\Delta_{\nu_{2}\PP V}]
.\]

\textbf{Skew-symmetric case}: Assume $n\geq 2$. 
The Grassmannian is the projection $ \Gr(2, V) = \pi_{\wedge}(X)$, which is a smooth homogeneous variety. We can conclude that $X^\vee \cap \PP \bw{2} V$ contains $\Gr(2, V)^\vee$ by Theorem~\ref{thm:tan}. The projection $\pi_\wedge$ of $\widehat T_x X$ can drop rank only when the base point is $0$, but $[0]$ is not a point of the ambient projective space, and the rank loci yield no additional information.  To understand the splitting of the skew-symmetric determinant, we need to know when $\Gr(2,V)^\vee$ is a hypersurface and determine the multiplicity.

A dimension count shows that $\Gr(2,V)^\vee$ is a hypersurface if and only if $n=\dim V$ is even and $n\geq 4$.  
Specifically, if $n$ is odd we get the additional condition that any hyperplane $H \subset \PP \bw{2}V$, viewed as a matrix, has (at least) a 1-dimensional kernel $K$. So, if $H$ is tangent to the Grassmannian $\Gr(2,V)$ at a point $E \in \Gr(2,V)$ we can assume that $K$ is not contained in $E$ because the dimension of the kernel must be odd. Then $H$ is  tangent along the line $\PP(\bigwedge^2\langle E , K\rangle)$. Thus $\Gr(2,V)^{\vee}$ is not a hypersurface, so  the discriminant is $0$ for odd $n$. 

To calculate the multiplicity in the even case, note that any skew-symmetric hyperplane that vanishes on $\widehat T_{a\wedge b}\Gr(2,V)$ also vanishes on
$\widehat T_{a\otimes b}\Seg(\PP V\times \PP V)$ and on 
$\widehat T_{b\otimes a}\Seg(\PP V\times \PP V)$, which are distinct when $a$ and $b$ are not colinear. So, Theorem~\ref{thm:tan} implies that we have at least $m=2$ and $\Delta_{\Gr(2,V)}^{2} \mid \Res(\Delta_{\Seg(\PP V \times \PP V)},\bigwedge^2 V^*)$. Further work, for instance computing the degrees, allows one to conclude that the multiplicity is exactly 2.
Thus, we find the well-known result that when $n$ is even, the square of the discriminant of a skew 2-form (the Pfaffian) is the determinant of a skew-symmetric matrix (up to scale):
\[
[\Res(\Delta_{\Seg(\PP V\times \PP V)},\bigwedge^2 V^*)] = [\Delta_{\Gr(2,V)}^{2}]. \qedhere
\]
\end{example}

\begin{example}\label{ex:gr36}
 Returning to the example in the introduction, consider a vector space $V$ with standard basis $\{e_{1},\ldots, e_{6}\}$ and subspace $E = \text{span}\{e_1,e_2,e_3\}$. The $G= \GL(V)$-orbit of  $\bw3E = e_{1}\wedge e_{2}\wedge e_{3}$ in $\PP \bw 3 V$ is the Grassmannian $\Gr(3,6)\subset \PP \bw3V$, with affine tangent space
\[
\widehat T_{E}\Gr(3,V) = \bw{3}E \oplus \left( E^{*} \otimes V/E \right)
.\]

Now consider a splitting $V  = V_{1} \oplus V_{2}\oplus V_{3}$ with $V_{i} = \langle e_{i}, e_{i+3}\rangle $ for $i=1,2,3$. The $\SL(V)$-module  $\bw{3} V$ splits as an $\SL(V_{1}) \times \SL(V_{2})\times \SL(V_{3})$-module as
\[
\bw{3} V = \left( \bigoplus_{1\leq i\neq j \leq 3  }\bw{2}V_{i}\otimes V_{j}  \right) \oplus \left( V_{1}\otimes V_{2}\otimes V_{3} \right),
\]
where we identify $\wedge$-products and $\otimes$-products of complementary spaces, and do not include $\bw{3}V_i$ since it is zero when $\dim V_i = 2$.  
The intersection of $\Gr(3,V)$ with $\PP(V_{1}\otimes V_{2}\otimes V_{3})$  is the Segre variety $\Seg( \PP V_{1} \times \PP V_{2}  \times \PP V_{3})$, a smooth homogeneous variety $G/P$.

This projection of $\Gr(3,6)$ may be interpreted as the flatlander's trifocal variety \cite{AST,Oeding_quad} or as the variety of principal minors of a $3\times 3$ matrix \cite{LinSturmfels}. In either case  the projection coincides with the entire ambient space $\PP(V_{1}\otimes V_{2}\otimes V_{3})$. So again we cannot apply Proposition~\ref{prop:gkz}, hence we use Theorem~\ref{thm:tan} instead.

We check the tangency condition. If we further decompose $\widehat T_{E}\Gr(3,V)$ as a module for the reductive part of $P$ we get 
$e_{1}\wedge e_{2}\wedge e_{3} \oplus \langle e_{i_{1}}\wedge e_{i_{2}} \wedge e_{k} \mid 1\leq i_{1} \neq i_{2} \leq 3, \;\;4\leq k \leq 6
 \rangle = $
{\begingroup
\scriptspace=-0.75pt
\begin{align*}
&=\langle
e_{1}\otimes e_{2}\otimes e_{3},\;
e_{1}\otimes e_{2}\otimes e_{6},\;
e_{1}\otimes e_{5}\otimes e_{3},\;
e_{4}\otimes e_{2}\otimes e_{3}
\rangle
\\
&\;\;\;\;\oplus
\langle
e_{6}\otimes e_{2}\otimes e_{3},\;
e_{5}\otimes e_{2}\otimes e_{3},\;
e_{1}\otimes e_{6}\otimes e_{3},\;
e_{1}\otimes e_{4}\otimes e_{3},\;
e_{1}\otimes e_{2}\otimes e_{5},\;
e_{1}\otimes e_{2}\otimes e_{4}
\rangle
\\&=
\underline{\left( e_{1}\otimes e_{2}\otimes e_{3} \right)}
\underline{\oplus
\left(V_{1}/e_{1} \otimes e_{2} \otimes e_{3} \right)\oplus
\left(e_{1} \otimes V_{2}/e_{2} \otimes e_{3} \right)\oplus
\left(e_{1} \otimes e_{2} \otimes V_{3}/e_{3} \right)
}
\\
&\;\;\;\; \oplus
\left(\big(V_{2}/e_{2} \oplus V_{3}/e_{3}\big) \otimes e_{2} \otimes e_{3}\right) \oplus
\left(e_{1} \otimes \big(V_{1}/e_{1} \oplus V_{3}/e_{3}\big) \otimes e_{3} \right)\oplus
\left(e_{1} \otimes e_{2} \otimes \big(V_{1}/e_{1}  \oplus V_{2}/e_{2} \big) \right)
.
\end{align*}
\endgroup
}
\hspace{-3pt}Note that $\widehat T_{e_{1}\otimes e_{2}\otimes e_{3}}\Seg (\PP V_{1}\times \PP V_{2} \times \PP V_{3})$ is a submodule (underlined), and that, moreover, the projection to $V_{1}\otimes V_{2}\otimes V_{3}$ is precisely this tangent space. Thus Theorem~\ref{thm:tan} applies, and (by a dimension calculation) we see that $\Delta _{\PP^{1}\times \PP^{1}\times \PP^{1}}$ divides the restriction of $\Delta_{\Gr(3,6)}$ to 3-mode binary tensors.
To look for more factors, we could compute the projection rank loci. However, the projection is $G = \GL(V_1)\times \GL(V_2) \times \GL(V_3)$-invariant and the Segre is the only $G$-variety in  $\PP (V_1\otimes V_2 \otimes V_3)$ whose dual is a hypersurface. So, no other factor can contribute. Further work (a degree comparison) can show that the multiplicity is one. 

Further identify $U = V_{1} = V_{2} = V_{3}$ to get a copy of $S^{3}U$ inside of $V_{1}\otimes V_{2}\otimes V_{3}$.
Note that the image of a general tangent space
$
\widehat T_{u_{1}\otimes u_{2}\otimes u_{3} }\Seg (\PP U^{\times 3})
$
under the projection to $S^{3}U$ is a tangent space of the Chow variety $\Chow_{1,1,1}\PP^{1}$, however, since every bivariate homogeneous form is completely decomposable, $\Chow_{1,1,1}\PP^{1} = \PP^{2}$, and we don't obtain any divisibility.

Again we cannot apply Proposition~\ref{prop:gkz}, but we can look at when the projection of
$
\widehat T_{u_{1}\otimes u_{2}\otimes u_{3} }\Seg (\PP U^{\times 3})
$ drops rank and check the tangency condition. 
First, when the lines $[u_{i}]$ are mutually distinct yields
$\Chow_{1,1,1}\PP^{1} = \PP^{2}$, already discussed. When two of the lines coincide the projection drops rank, and one obtains $\Chow_{2,1}\PP^{1}$ as a rank locus. However, its dual is not a hypersurface (\cite[Thm.~1.3]{oeding_hyperdet}), and thus does also not contribute to the divisibility.  
Finally if we have a coincidence of lines $[u_{1}] = [u_{2}] = [u_{3}]$, the resulting tangent space projects onto the tangent space of a point of the Veronese variety $\nu_{3}\PP U$. Hence the Veronese satisfies the tangency condition and Theorem~\ref{thm:tan} applies, so we see that $\Delta _{\nu_{3}\PP^{1}}$ divides  $\Res(\Delta_{\PP^{1}\times \PP^{1}\times \PP^{1}}, S^3 \CC^2)$. 
Since all these invariants have degree 4, the division relations are equalities (up to scalar). 
\end{example}

\begin{example}\label{ex:curve}
The following example was suggested by a referee. It is nice in that is not so symmetric, illustrating further features of what can happen to the restriction of the dual. Our computations are performed in Macaulay2 \cite{M2}, aided by the package \texttt{Resultants} \cite{Stagliano}.
Let $C$ denote the smooth curve in $\PP^3$ respectively with ideal and discriminant 
\[\mathcal{I}(C) = \langle x z-y q, z^{3}-x q^{2}-z q^{2}, y z^{2}-x^{2}q-y q^{2}, x^{3}-y^{2}z+x y q \rangle \subset \CC[x,y,z,q]
,\] 
\[\begin{smallmatrix}
\Delta_C = -4 x^{6}+8 x^{4}y^{2}-4 x^{2}y^{4}+12 x^{5}z-52 x^{3}y^{2}z+8 x y^{4}z-12 x^{4}z^{2} 
 +107 x^{2}y^{2} z^{2}-4 y^{4}z^{2}+4 x^{3}z^{3}-90 x y^{2}z^{3}+27 y^{2}z^{4} \\
 -12 x^{4}y q+60 x^{2}y^{3}q-16 y^{5}q+6 x^{3}y z q-208 x y^{3}z q+6 x^{2}y z^{2}q 
 +144 y^{3}z^{2}q+27 x^{4}q^{2}-48 x^{2}y^{2}q^{2}+128 y^{4}q^{2}+192 x y^{2}z q^{2}-256 y^{3}q^{3}.
 \end{smallmatrix}\]
The projection of $C$ from $Q:=\{[0:0:0:1]\}$ is the nodal curve $N$ with ideal $\langle x^{3}+x^{2}z-y^{2}z \rangle $. The jacobian of $N$ has the following primary decomposition:
\[
\langle y,\,x \rangle
 \cap
  \langle z^{2},\,y\,z,\,3\,y^{2}+2\,x\,z,\,3\,x^{
       2}+2\,x\,z
       \rangle
.\]
Since the second ideal has support $[0:0:0]$ we conclude that $N$ is smooth away from the node  $n :=[0:0:1] \in N \subset \PP^2$. 
The dual variety $N^\vee$ in $\PP^2$ has discriminant
\[
\Delta_N = -4x^{4}+8x^{2}y^{2}-4y^{4}+4x^{3}z-36xy^{2}z+27y^{2}z^{2}.
\]
The restriction of $\Delta_C$ to the space $q=0$ factors as 
\[
\Res(\Delta_C, \{q=0\}) = \left(x-z\right)^{2}\left(-4 x^{4}+8 x^{2}y^{2}-4 y^{4}+4 x^{3}z-36 x y^{2}z+27 y^{2}z^{2}\right).\]
Let us compute the projection rank loci. Since $N$ is the projection of $C$, $N$ satisfies the tangency condition by Corollary~\ref{cor:proj}.  The projection drops rank when $\widehat T_pC$ intersects the center of projection. The tangent space of $C$ at $p$ is the kernel of the jacobian $JC_p$. So we can compute the projection rank locus by solving the system of equations:
$JC_p \cdot Q = 0$. We obtain an ideal with primary decomposition
\[
\langle q,\,y,\,x^{2}\rangle\,\cap \, \langle y,\,x+z,\,z^{2}\rangle.
\]
These primary components intersected with $C$ are supported only at the cone point and at $p_0 = [0:0:0:1]$. 
So the rank locus consists only of $\pi_Q \widehat T_{p_0}C \cap N$. Of course the center of projection $Q$ does not project into $N$, but we can still get a rank locus. Let $\langle \rangle$ denote the column space when surrounding a matrix. We compute
\[\widehat T_{p_0}C=  
\left \langle
\left(\begin{smallmatrix}
       \vphantom{\left\{-1\right\}}-1&0\\
       \vphantom{\left\{-1\right\}}0&0\\
       \vphantom{\left\{-1\right\}}1&0\\
       \vphantom{\left\{-1\right\}}0&1\\
       \end{smallmatrix}\right)\right \rangle 
\quad \text{and} \quad 
\pi_Q \widehat T_{p_0}C = 
\left \langle \left(\begin{smallmatrix}
       \vphantom{\left\{-1\right\}}-1\\
       \vphantom{\left\{-1\right\}}0\\
       \vphantom{\left\{-1\right\}}1\\
       \end{smallmatrix}\right)
\right \rangle .\]
So  $\widehat T_{p_0}C$ projects to the point $e = [1:0:-1] \in N$ and the rank drops. Note also that 
\[\widehat T_e N = 
\left \langle
 \left(\begin{smallmatrix}
       0&-1\\
       1&0\\
       0&1\\
       \end{smallmatrix}\right)
\right \rangle        
       .\]
Now consider the point $e$ as a variety, denoted $E$ for disambiguation. Since $\pi_Q \widehat T_{p_0} C = \widehat  T_eE$, we see that $E$ satisfies the tangency condition, and  hence the dual $E^\vee$, which is equal to the line $\V(x-z)$, must contribute to the restriction of the restriction of the dual of $C$.  Now lift $\widehat T_eE$ to  $\hat e = [x:0:-x:q]$, which is equal to $\widehat T_{p_0}C$. Note that $\hat e$ intersects $C$ only when $x =0$, i.e. at the point $ p_0$.  Moreover, the ideal of $C$ restricted to $\hat e$ becomes $\langle  -x^{2},\,-x^{3},\,-x^{2}q,\,x^{3} \rangle$, hence the line $\hat e$ osculates $C$ to order 2 at $p_0$. 
Therefore all hyperplanes in $\PP^2$ that are tangent to $E$ at $e$ must lift to hyperplanes in $\PP^3$ that are tangent to order $2$ to $C$ at $p_0$ , and hence $x-z$ must divide the restriction twice.

In addition to $p_0$ we're interested in what happens above the node, namely at
\[
p_{1} = [0:0:1:1],\quad p_{-1} = [0:0:-1:1] .\]
The respective tangent cones $\widehat{T_{p_i}}C$ are the column spaces of the following matrices.
\[
 \widehat T_{p_1}C= 
\left \langle
\left(\begin{smallmatrix}
       \vphantom{\left\{-1\right\}}2&0\\
       \vphantom{\left\{-1\right\}}2&0\\
       \vphantom{\left\{-1\right\}}1&1\\
       \vphantom{\left\{-1\right\}}0&1\\
       \end{smallmatrix}\right)
       \right \rangle
       , \quad
 \widehat T_{p_{-1}}C= 
\left \langle
 \left(\begin{smallmatrix}
       \vphantom{\left\{-1\right\}}2&0\\
       \vphantom{\left\{-1\right\}}-2&0\\
       \vphantom{\left\{-1\right\}}1&-1\\
       \vphantom{\left\{-1\right\}}0&1\\
       \end{smallmatrix}\right)        \right \rangle. 
\]
Note that  $p_1$ and $p_{-1}$ of $C$ project to $n$, and moreover the tangent lines near $n$  approach the projections   $\pi_Q \widehat T_{p_1}C$ and $\pi_Q \widehat T_{p_{-1}}C$. These projections don't drop rank, however. 
\end{example}

\section{Restrictions of discriminants of Lie algebras}\label{sec:lie}
We establish division relations based on Theorem~\ref{thm:tan} between different equations of dual varieties. Those relations will allow us to get new explicit equations of duals.

Recall that for $\g$ a semi-simple Lie algebra, the \emph{adjoint variety}, denoted $X_G^{\text{ad}}$, is the projectivization of the highest weight orbit in $\g$ for the adjoint action of the Lie group $G$,
\begin{equation}
 X_G^{\text{ad}}=\PP(G.v)\subset \PP \g,
\end{equation}
with $v$ a highest weight vector of $\g$.
The adjoint variety $X_G^{\text{ad}}$ is the unique closed orbit for the adjoint action on $\PP \g$. 
More generally suppose $V_\lambda$ is an irreducible $G$-module with highest weight $\lambda$, and highest weight vector $v_\lambda \in V_\lambda$. Consider the \emph{homogeneous variety} $G.v_\lambda  = G/P$, where $P$ is the stabilizer of $v_\lambda \in V_\lambda$. 
Then \cite[Claim~23.52]{FultonHarris} says that $G/P$ is the unique closed orbit of $G$ acting on $V_\lambda$. Duals of adjoint varieties are hypersurfaces \cite{TevelevJMS}, and as such, are defined by a single (up to scale) irreducible homogeneous polynomial,  the $G$-\emph{discriminant}, denoted $\Delta_G$ (instead of $\Delta_{X^{\text{ad}}_G}$). 

Suppose $\g$ is equipped with a $\ZZ_k$-grading 
\begin{equation}\label{eq:grading}
\g=\g_0\oplus \g_1\oplus\dots\oplus \g_{k-1}
\end{equation} and $\g_0$ is such that $\text{Exp}(\g_0) =:G_0$ is a connected component containing the identity of $G$ so that $G_{0}$ is the (Lie) subgroup of $G$ with Lie algebra $\g_{0}$. Since $\g_0$ acts on $\g$ and preserves the grading (i.e. $[\g_0,\g_i] \subset \g_i$) then \eqref{eq:grading} is a $G_0$-module decomposition. One can establish relations between duals of $G_0$-orbits and the restriction of the $G$-discriminant by Theorems~\ref{thm:tan}. 

\begin{theorem} \label{thm:g0}
Use the notation above. Suppose $ \PP \g_s \cap X^\text{ad}_G \neq \emptyset$ and let $[v_\lambda] $ be one such point of this intersection. The respective homogenous varieties $Y = G_0.[v_\lambda] \subset \PP \g_s$ and $X^\text{ad}_{G}  = G. [v_\lambda] \subset \PP \g$ satisfy the tangency condition, hence: 
$Y^\vee \subset (X^\text{ad}_G)^\vee \cap \PP (\g_s^*)$. If $Y^\vee$ is a hypersurface, then $(\Delta_{Y})^m \mid \Res(\Delta_{X^\text{ad}_G}, \g_s)$, for some integer $m>0$.
\end{theorem}
\begin{proof}
Let $A = \g_s$ with $s\neq 0$ and  $B=\bigoplus_{i\neq s}\g_i$.
By assumption we have $0\neq v_\lambda \in  \g_s \cap \widehat X^\text{ad}_G $. 
Because $X^\text{ad}_G$ is a highest weight orbit we may choose a set of positive roots of $G$, which also fixes the positive roots of $G_0$, so that $v_\lambda$ is a highest weight vector of $\g$ as a $G$-module, and therefore a highest weight vector of $\g_s$ as a $G_0$-module.  Hence $Y$ is homogenous by \cite[Claim~23.52]{FultonHarris}, and in particular, it is closed. 

Because $Y$ is homogeneous, it suffices to verify the tangency condition for $Y$ at $v_\lambda$. Notice that the grading \eqref{eq:grading} is a model for the Lie algebra $\g$,
in particular the brackets with elements of the Lie algebra $\g_0$ define a $\g_0$-action on each $\g_s$. So the tangent space of $G_0.v_\lambda$ is computed by the bracket with the Lie algebra $\g_0$:
\[
\widehat T_{v_\lambda} Y = [\g_0,v_\lambda] \subset \g_s = A,
\]
where $[,]$ denotes Lie bracket. The tangent space upstairs respects the grading: 
\begin{equation}\label{eq:tx}
\widehat T_{v_\lambda} X^\text{ad}_G =  [\g,v_\lambda]= [\g_0,v_\lambda] \oplus \bigoplus_{i =1}^{k-1} [\g_i,v_\lambda].
\end{equation}
By the grading, the sum is a direct sum and the only part that intersects $\g_s = A$ is the first one listed on the RHS of \eqref{eq:tx}.
Hence the projection is precisely $ \pi_B(\widehat T_{v_\lambda} X^\text{ad}_G) = \widehat T_{v_\lambda} Y$, thus $Y$ satisfies the tangency condition. The remainder of the conclusions are from Corollary~\ref{thm:tan}.
\end{proof}
\begin{example}\label{ex:e8}
Consider the Lie algebra $\mathfrak{e}_8$ with the following $\ZZ_3$-grading (see \cite{Vinberg-Elasvili, Katanova}):
\[
 \mathfrak{e}_8=\bw3 V^*\oplus \sl_9\oplus \bw3 V = \g_{-1}\oplus \g_{0}\oplus \g_{1},
\]with $\dim V = 9$. 
Let $E=e_1\wedge e_2\wedge e_3 \in \bw3 V$.  $E$ is nilpotent in $\mathfrak{e}_8$ and a calculation on the roots shows that $\text{dim}([\mathfrak{e}_8,E])=58$, i.e., $\PP(\text{E}_8.E)$ is a nilpotent orbit of dimension $57$ and thus corresponds to $X_{\text{E}_8}^{\text{ad}}$. In particular, $0\neq E \in \widehat X^\text{ad}_{\text{E}_8} \cap \g_1$. 
The $G_0=\SL(V)$-orbit of $E$ in $\PP(\bw3 V)$ is the Grassmannian $\Gr(3,9) = \Gr(3,V) \subset \PP(\bw3 V)$. Then, by Theorem~\ref{thm:g0} we obtain the divisibility relation $ \Delta_{\Gr(3,9)} \mid \Res(\Delta_{\text{E}_8},\g_1^*)$. 
\end{example}
\begin{example}\label{ex:e7}
Consider $\mathfrak{e}_7$ with the following $\ZZ_2$-grading (\cite{Katanova}):
\[
  \mathfrak{e}_7=\mathfrak{s}\mathfrak{l}_8 \oplus \bw 4 \CC^8 =: \g_0 \oplus \g_1.
 \]
Restrict $\Delta_{\text{E}_7}$ to $\g_1^*$. Set $E =e_1\wedge e_2\wedge e_3\wedge e_4$.  Note that $G_0.[E]=\Gr(4,8) \subset \PP \g_1$ and check like in Example~\ref{ex:e8} that $[E] \in X_{\text{E}_7}^{\text{ad}}\cap \PP(\bw 4 \CC^8)$.
 By Theorem~\ref{thm:g0} and comparing degrees, $\text{deg}((X_{\text{E}_7}^{\text{ad}})^\vee)=126$ \cite{TevelevJMS} and $\text{deg}\Gr(4,8)^\vee=126$ \cite{LascouxDegreeDualGrassmannian},  we conclude
\begin{equation}
[ \Delta_{\Gr(4,8)}]  = [\Res(\Delta_{\text{E}_7},\bw 4{\CC^8}^*)].
\lqedhere
\end{equation}\label{eq:e7}
\end{example}
\begin{example}\label{ex:e6}
 Consider $\mathfrak{e}_6$ with the following $\ZZ_3$-grading:
\[
  \mathfrak{e}_6=((\CC^3)^*\otimes (\CC^3)^*\otimes (\CC^3)^*)\oplus (\sl_3\oplus\sl_3\oplus\sl_3)\oplus( \CC^3\otimes\CC^3\otimes \CC^3) =:\g_{-1} \oplus \g_0 \oplus \g_1.
  \]
Restrict $\Delta_{\text{E}_6}$ to $(\CC^3\otimes\CC^3\otimes \CC^3)^* = \g_1^*$. 
Set $E  =e_1\otimes e_1 \otimes e_1$ and check that $[E] \in X^\text{ad}_{\text{E}_6} \cap \g_1$. 
Note that $\Seg(\PP^2\times\PP^2\times \PP^2) = G_0. [E] \subset \PP \g_1$.  By Theorem~\ref{thm:g0} 
\[\Delta_{\PP^2 \times \PP^2 \times \PP^2} \mid \Res(\Delta_{\text{E}_6},(\CC^3\otimes\CC^3\otimes \CC^3)^*). \qedhere
\]
\end{example}
\begin{example}[\cite{holweck_4qubit2}]\label{ex:so8}
Consider $\so_8$ with the following $\ZZ_2$-grading:
\[
  \so_8=\left(\sl_2\oplus \sl_2\oplus \sl_2\oplus \sl_2\right)\oplus \CC^2\otimes\CC^2\otimes \CC^2 \otimes \CC^2.
 \]
Restrict $\Delta_{\SO_8}$ to $((\CC^2)^{\otimes 4})^* = \g_1^*$. Set $E = e_1\otimes e_1 \otimes e_1\otimes e_1$ and check that   $[E] \in X^\text{ad}_{\text{SO}_8} \cap \g_1$. 
Note that $\Seg(\PP^1\times\PP^1\times \PP^1\times \PP^1) = G_0.[E] \subset \PP \g_1 $.  By Theorem~\ref{thm:g0} and checking degrees,
 $\text{deg}((X_{\SO(8)}^{\text{ad}})^\vee)=24$ and $\text{deg}((\Seg(\PP^1\times\PP^1\times\PP^1\times\PP^1))^\vee)=24$, we obtain
\[
[ \Delta_{\PP^1\times\PP^1\times\PP^1\times \PP^1}] = [ \Res(\Delta_{\SO(8)},(\CC^2)^{\otimes 4 *})].\qedhere\]
\end{example}

A general philosophy of this paper is to show that new discriminants and some well-known nontrivial hyperdeterminants can be derived by restricting the $\text{E}_8$ discriminant, and as such, they are a flavor of sparse resultant.
 Figure~\ref{fig:table} gives a summary, where we use $\big|$ to denote ``divides,'' and  $\Gamma^{\text{ad}}_G$ to denote the adjoint representation of $G$.

\begin{figure}[!h]
\resizebox{1\linewidth}{!}{
$\xymatrix{\Delta_{\text{E}_8} \ar[r] \ar[d]
& \Delta_{\text{E}_7} \big| \Res(\Delta_{\text{E}_8},(\Gamma^{\text{ad}}_{\text{E}_7})^*) \ar[r]\ar[d]
& \Delta_{\text{E}_6} \big|  \Res(\Delta_{\text{E}_7},(\Gamma_{\text{E}_6}^{\text{ad}})^*)\ar[d]\ar[r] 
& \Delta_{\SO(8)} \big| \Res(\Delta_{\text{E}_6},(\Gamma_{\SO_8}^{\text{ad}})^*) \ar[d]\\
  \Delta_{\Gr(3,9)} \big| \Res(\Delta_{\text{E}_8},((\bigwedge^{\hspace{-.2em}^3} \CC^9)^*)) \ar[d]
  & \Delta_{\Gr(4,8)} \big| \Res(\Delta_{\text{E}_7},(\bigwedge^{\hspace{-.2em}^4}  \CC^8)^*) \ar[d]
  &  \Delta_{(\PP^2)^{\times 3}} \big|\Res(\Delta_{\text{E}_6},(\CC^3)^{\otimes 3*})\ar[d]^{\pi_{(S^3(\CC^3)^*)^\perp}} 
  &\Delta_{(\PP^1)^{\times 4}}\big|\Res(\Delta_{\SO(8)},(\CC^2)^{\otimes 4*})\ar[d]\\
 \Delta_{(\PP^2)^{\times3}}\big|\Res(\Delta_{\Gr(3,9)},(\CC^3)^{\otimes 3*}) \ar[d]
 & \Delta_{(\PP^1)^{\times 4}} \big| \Res(\Delta_{\Gr(4,8)},(\CC^2)^{\otimes 4*}) \ar[d]
 & \Delta_{v_3(\PP^2)} \big| \Res(\Delta_{(\PP^2)^{\times3}},S^3(\CC^3)^*) 
 & \Delta_{v_4(\PP^1)} \big| \Res(\Delta_{(\PP^1)^{\times 4}},S^4(\CC^2)^*)\\
 \Delta_{v_3(\PP^2)} \big|\Res(\Delta_{(\PP^2)^{\times3}},S^3(\CC^3)^*)
 &\Delta_{v_4(\PP^1)} \big| \Res(\Delta_{(\PP^1)^{\times 4}},S^4(\CC^2)^*)}
 $
 }
\caption{Division relations for a sequence of discriminants starting from the $\text{E}_8$-discriminant. The first row comes from the inclusion $\so_8\subset \e_6\subset \e_7\subset \e_8$, and an application of \cite[Thm~2.5]{TevelevJMS}. The second row comes from Examples~\ref{ex:e8}, \ref{ex:e7}, \ref{ex:e6}, and \ref{ex:so8}. The last two rows will be explicitly given in Section~\ref{sec:gr39gr48}.}\label{fig:table}
\end{figure}

\section{Expressions of the restriction of $\text{E}_n$-discriminants to semi-simple elements}\label{sec:semi-simple}
{In this Section we restrict the discriminants $\Delta_{\text{E}_8}, \Delta_{\text{E}_7}$, and $\Delta_{\text{E}_6}$ to the semi-simple elements of a specific $G$-module when $G$ is a subgroup of $\text{E}_8$, $\text{E}_7$, or $\text{E}_6$.} Thanks to the Killing form one identifies $\g\simeq\g^*$ and thus the discriminants$ \Delta_{\text{E}_n}$, for $n=6,7,8$ are elements of $\CC[\mathfrak{e}_n]^{\text{E}_n}$.
Recall the Chevalley Restriction Theorem that $\CC[\g]^{G} = \CC[\mathfrak{h}]^{W}$, where $\g$ is a complex semi-simple Lie algebra associated with the Lie group $G$,  $\mathfrak{h} \subset \g$ is a Cartan subalgebra and $W$ is the Weyl group.
The \emph{Jordan-Chevalley decomposition} (or \emph{Jordan decomposition} for short) of an element $x \in \g$ is a unique decomposition of form $x=s+n$, with $s$  semi-simple and $n$ nilpotent. We will fix a Cartan subalgebra $\mathfrak{h}$ with basis $t_{1},\ldots, t_{n}$ and declare a \emph{generic semi-simple form} to be $\mathscr{s} = \sum_{i}s_{i}t_{i}$, for parameters $s_{i}$ not all zero. 
Our interest in expressing invariants on a generic semi-simple form $\mathscr{s}$ is the following.
The fundamental invariants, being continuous, take the same value on $s+n$ as they do on $s$. This was noted in \cite[Section~3.1]{Vinberg-Elasvili} and stated more explicitly in \cite[Prop.~2.2]{SarosiLevay}. Thus, we can evaluate the invariants on a generic semi-simple form, as this is essentially no loss of information for computations in the invariant ring, and moreover, the formulas become very nice.
For instance, Tevelev showed that the discriminant of the unique closed adjoint orbit restricted to $\mathfrak{h}$ is the product of the roots \cite[Thm~2.5]{TevelevJMS}. 

Finally, we will be interested in expressing hyperdeterminants and discriminants in terms fundamental invariants because this will allow for evaluation without computing a Jordan-Chevalley decomposition. To determine the expression of an invariant as a polynomial in fundamental invariants, it suffices to work on a basis of semi-simple elements and work on a generic form that parametrizes a non-trivial open subset in that Cartan subalgebra.%
\begin{prop}\label{prop:up}
Suppose $\g$  is  a complex semi-simple Lie algebra with fixed Cartan subalgebra $\mathfrak{h}$. Let $\mathscr{s}\in \mathfrak{h}$ denote a generic semi-simple form. Suppose $f_{1},\ldots f_{s} $ is a $G$-invariant generating set of the invariant ring $\CC[V]^{G}$. If $h\in \CC[V]^{G}$ is an invariant polynomial and $h(\mathscr{s})= \sum_{I} a_{I}f^{I}(\mathscr{s})$
 then
$
h = \sum_{I} a_{I}f^{I}
$.
\end{prop}
\begin{proof}
This follows from the fact that $\mathscr{s}$ parametrizes a dense open subset of $\mathfrak{h}$, so $h(\mathscr{s})= \sum_{I} a_{I}f^{I}(\mathscr{s})$ implies that this formula holds in $\CC[\mathfrak{h}]^{W}$, and the Chevalley Restriction Theorem implies that the formula holds in  $\CC[\g]^{G}$. 
\end{proof}

\subsection{Restricting  $\Delta_{\text{E}_{8}}$ to semi-simple elements of $\bigwedge^3 \CC^9$}
Our basic reference is \cite{Vinberg-Elasvili}. 
We will use Tevelev's result that the restriction of $\Delta_{\text{E}_8}$ on a Cartan subalgebra of $\mathfrak{e}_8$ is the product of the roots: $\Res(\Delta_{\text{E}_8},\mathfrak{h})=\prod_{\alpha \in R} \alpha$, where $R$ is the set of roots of $\mathfrak{e}_8$ \cite[Thm~2.5]{TevelevJMS}.
Again, we will use the $\ZZ_3$-grading of the exceptional Lie algebra $\mathfrak{e}_{8}$ in \cite{Vinberg-Elasvili}:
\[
\mathfrak{e}_{8} = \bw{3}V^{*}\oplus \mathfrak{sl}(V)\oplus \bw{3}V = \g_{-1} \oplus \g_{0}\oplus \g_{1}
.\]
A choice of a Cartan subalgebra of $\mathfrak{e}_8$ induces weights. A standard construction that respects the $\ZZ_3$ grading is as follows.
Let $e_i$ denote a basis of $V$. Define \emph{weights} $\varepsilon_{i}$ by $\varepsilon_{i}(e_{j}) = \delta_{i,j}$. 
Define $\varepsilon$ to be additive over tensor products and invariant under non-zero scalar multiplication.
The induced basis vectors for $\bw{3}V$ are $e_{ijk} := e_{i}\wedge e_{j}\wedge e_{k}$, which have weights $\varepsilon_{i}+ \varepsilon_{j}+ \varepsilon_{ k}$. 

Let $\mathfrak{h}\subset  \g_{0}$ be the standard Cartan subalgebra of traceless diagonal matrices in $\sl(V)$. In bases, $\mathfrak{h}$ is the hyperplane  $\sum_{i}e_{i}\otimes e^{i}=0$ in the span of $e_{1}\otimes e^{1}, e_{2}\otimes e^{2},\ldots,e_{9}\otimes e^{9}$.

Contraction with the standard volume form $\Omega = e_1\wedge \cdots \wedge e_9$ gives the duality $\bw{k}V \cong \bw{9-k}V \cong \bw{k}V^*$. Requiring the volume form to have weight zero implies $\sum \varepsilon_i =0$, hence dualizing $k$-forms negates their weights. Let $e^i$ denote the dual basis vectors of $V^*$ of weights $-\varepsilon_i$. The induced dual basis vectors $e_{ijk}^*$ for $\bw{6} V \cong \bw{3}V^*$ have weights $-(\varepsilon_{i}+ \varepsilon_{j}+ \varepsilon_{ k})$. 
The standard basis $e_i\otimes e^j$ on $V\otimes V^*$ gives a weight basis on $\sl(V)$ with weights $\varepsilon_i-\varepsilon_j$ (the basis vectors on the diagonal have weight $0$). 
The $\text{E}_8$ root system is
\[\Sigma = \{ \varepsilon_{i}-\varepsilon_{j},  \pm(\varepsilon_{i}+\varepsilon_{j}+\varepsilon_{k}) \} \quad (i,j,k \text{ distinct}).
\]

Now we construct an alternative Cartan subalgebra of $\mathfrak{e}_8$.
Let $\mathfrak{s}_{1} \subset \g_{1}$ be spanned by
\begin{align}\notag
p_1=e_{123}+e_{456}+e_{789},
&\phantom{\quad }& p_2=e_{147}+e_{258}+e_{369},\\
p_3=e_{159}+e_{267}+e_{348}, & &
p_4=e_{168}+e_{249}+e_{357}.\label{eq:ssb39}
\end{align}
Let $p_i^*$ denote the dual vector to $p_i$, obtained by contraction with $\Omega$. 
 The following set of \emph{basic semi-simple elements} of $\mathfrak{e}_{8}$,
\[
\{p_{i}^{*},p_{i}\mid 1\leq i\leq 4\}, 
\]
 forms a basis of a maximally commuting subalgebra $\mathfrak{s}_{-1}\oplus  \mathfrak{s}_{1} = \mathfrak{s} \subset \mathfrak{e}_{8}$.
 An open set of semi-simple elements in $\mathfrak{s}_{1} \subset \bw{3}V = \g_{1}$ is given by the generic form
\[p = z_{1}{p}_{1}+ z_{2}{p}_{2}+ z_{3}{p}_{3}+ z_{4}{p}_{4},
\]
for parameters $z_{i}$ not all zero.
 Let $\omega = e^{2i\pi/3 }$.
 The $240$ roots of $\text{E}_8$ restrict to $\mathfrak{s}_{1}$ via
\[\begin{smallmatrix}
3\,\varepsilon_{{1}}(p)&=&\phantom{\omega}z_{1}+\phantom{\omega}z_{2}+\phantom{\omega}z_{3}+\phantom{\omega}z_{4},&&
3\,\varepsilon_{{2}}(p)&=&\phantom{\omega}z_{1}+\omega z_{2}+\omega z_{3}+\omega z_{4}, &&
3\,\varepsilon_{{3}}(p)&=&\phantom{\omega}z_{1}+{\overline \omega}z_{2}+{\overline \omega}z_{3}+{\overline \omega}z_{4}, \\
3\,\varepsilon_{{4}}(p)&=&\omega z_{1}+\phantom{\omega}z_{2}+{\overline \omega}z_{3}+\omega z_{4},&&
3\,\varepsilon_{{5}}(p)&=&\omega z_{1}+\omega z_{2}+\phantom{\omega}z_{3}+{\overline \omega}z_{4},&&
3\,\varepsilon_{{6}}(p)&=&\omega z_{1}+{\overline \omega}z_{2}+\omega z_{3}+\phantom{\omega}z_{4},\\
3\,\varepsilon_{{7}}(p)&=&{\overline \omega}z_{1}+\phantom{\omega}z_{2}+\omega z_{3}+{\overline \omega}z_{4},&&
3\,\varepsilon_{{8}}(p)&=&{\overline \omega}z_{1}+\omega z_{2}+{\overline \omega}z_{3}+\phantom{\omega}z_{4}, &&
3\,\varepsilon_{{9}}(p)&=&{\overline \omega}z_{1}+{\overline \omega}z_{2}+\phantom{\omega}z_{3}+\omega z_{4}. 
\end{smallmatrix}\]
The product of these restricted roots is the restriction of $\Delta_{\text{E}_8}$ to the generic semi-simple tensor in $\bw{3}\CC^9$. Collecting like terms in the product we have
$[\Res(\Delta_{\text{E}_{8}},\mathfrak{s}_{1})] =[ h^{6}]$, with $h =$
{\bgroup
\makeatletter
\renewcommand{\maketag@@@}[1]{\hbox{\m@th\normalsize\normalfont#1}}%
\makeatother
\fontsize{8}{9}\selectfont
\medmuskip=-1mu
\thinmuskip=-1mu
\thickmuskip=-1mu
\nulldelimiterspace=-1pt
\scriptspace=-0.75pt
\begin{align}&
z_{4}(z_{1}-z_{2}+z_{3})(z_{1}+\omega z_{2}+z_{3})
(z_{1}-z_{2}-\omega z_{3})
(z_{1}+\overline \omega z_{2}+z_{3})
(z_{1}-z_{2}-\overline \omega z_{3})
(z_{1}+\omega z_{2}- \omega z_{3})
(z_{1}+\omega z_{2}-\overline \omega z_{3})
(z_{1}+\overline \omega z_{2}-\omega z_{3})
(z_{1}+\overline \omega z_{2}-\overline \omega z_{3})
\nonumber \\[-.5ex]&
\cdot z_{3}(z_{1}+z_{2}-z_{4})
(z_{1}-\omega z_{2}-z_{4})
(z_{1}+z_{2}+\omega z_{4})
(z_{1}-\overline \omega z_{2}-z_{4})
(z_{1}+z_{2}+\overline \omega z_{4})
(z_{1}-\omega z_{2}+ \omega z_{4})
(z_{1}-\overline \omega z_{2}+\omega z_{4})
(z_{1}-\omega z_{2}+\overline \omega z_{4})
(z_{1}-\overline \omega z_{2}+\overline \omega z_{4})
\nonumber \\[-.5ex]&
\cdot z_{2}(z_{1}-z_{3}+z_{4})
(z_{1}+\omega z_{3}+z_{4})
(z_{1}-z_{3}-\omega z_{4})
(z_{1}+\overline \omega z_{3}+z_{4})
(z_{1}-z_{3}-\overline \omega z_{4})
(z_{1}+\omega z_{3}- \omega z_{4})
(z_{1}+\omega z_{3}-\overline \omega z_{4})
(z_{1}+\overline \omega z_{3}-\omega z_{4})
(z_{1}+\overline \omega z_{3}-\overline \omega z_{4})
\nonumber \\[-1.5ex]&
\cdot z_{1}(z_{2}+z_{3}+z_{4})
(z_{2}-\omega z_{3}+z_{4})(z_{2}+z_{3}-\omega z_{4})
(z_{2}-\overline \omega z_{3}+z_{4})(z_{2}+z_{3}-\overline \omega z_{4})
(z_{2}-\omega z_{3}-\omega z_{4})
(z_{2}-\overline \omega z_{3}- \omega z_{4})
(z_{2}-\omega z_{3}-\overline \omega z_{4})
(z_{2}-\overline \omega z_{3}-\overline \omega z_{4}).\label{eq:gr39z}
\end{align}
\egroup }
\hspace{-3pt}\noindent
 We will use this expression of $\Res(\Delta_{\text{E}_8}, \mathfrak{s}_1)$ in Section~\ref{sec:fundamental} when we establish the expression of $\Delta_{G(3,9)}$ on fundamental invariants.

\subsubsection{Restricting $\Delta_{\text{E}_8}$ and $\Delta_{\text{E}_6}$ to semi-simple elements of $\CC^3\otimes\CC^3\otimes \CC^3$} Choose an ordered basis of $\CC^{9}$ in 3 triplets coinciding with the choices determining the expression of $p_{1}$ in \eqref{eq:ssb39}:
\[
\langle e_{1}, e_{4}, e_{7} \rangle  \oplus
\langle e_{2}, e_{5}, e_{8} \rangle  \oplus
\langle e_{3}, e_{6}, e_{9} \rangle
  = \CC^{9}
.\]
Now we consider a copy of $\CC^{3}\otimes \CC^{3}\otimes \CC^{3}$ inside $\bw{3}\CC^{9}$ induced by this choice, and take $\{x_{I} \mid I\in \{0,1,2\}^{3}\} $ to be the corresponding basis of the tensor space, where $x_{I} = x_{i_{1}}\otimes x_{i_{2}}\otimes x_{i_{3}}$ and
a 0 (respectively a 1, or 2) in position $j$ corresponds to the first (respectively the second, or third) basis vector in factor $j$.
This induces a choice of basic semi-simple elements  $\mathfrak{s}' \subset \CC^{3}\otimes \CC^{3}\otimes \CC^{3}$:
\[\begin{matrix}
p_{2} &=& x_{000}+x_{111} + x_{222}, &&
p_{3} &=& x_{012}+x_{201} + x_{120}, &&
p_{4} &=& x_{021}+x_{102} + x_{210}.
\end{matrix}
\]
An open subset of semi-simple elements in $\mathfrak{s}'$ is given by the generic form
$
bp_{2} + cp_{3}+ dp_{4},
$ for some constants\footnote{We choose $b,c,d$ instead of $z_2,z_3,z_4$ to emphasize that we're in a different ambient space.} $b,c,d$ not all zero.
Evaluating fundamental invariants and the hyperdeterminant of $3\times 3\times 3$ tensors on $\mathfrak{s}'$ produces formulas that can be found in \cite[Eq.~(5)]{BremnerHuOeding}. 
 When restricted to $\mathfrak{s}'$ (by setting $a =z_1=0$, $b = z_2$, $c=z_3$, $d=z_4$) the product of the roots of $\text{E}_{8}$ becomes zero because of the factor of $z_1^6$ in \eqref{eq:gr39z}, which gets sent to zero in the projection. Restricting  \eqref{eq:gr39z} to a generic element of $\mathfrak{s}'$ and collecting like terms we obtain:
\[
\left[{ \small \Res\left(\mfrac{1} {z_1^{6}}(\Res(\Delta_{\text{E}_{8}},\mathfrak{s}_1)), \mathfrak{s}'\right) } \right]
= \left[
\begin{smallmatrix}
 \left( b+c+d \right) ^{6}
\left( {b}^{2}-bc+2\,bd+{c}^{2}-cd+{d}^{2} \right) ^{6}
\left( {b}^{2}-bc-bd+{c}^{2}+2\,cd+{d}^{2} \right) ^{6} \\
\cdot \left( {b}^{2}-bc-bd+{c}^{2}-cd+{d}^{2} \right) ^{6}
 \left( {b}^{2}+2\,bc-bd+{c}^{2}-cd+{d}^{2} \right) ^{6}
 {b}^{6}{c}^{6}{d}^{6}
 \\ 
\cdot  \left( c-d \right) ^{18}  \left( b-d \right) ^{18}\left( b-c \right) ^{18}
 \left( {c}^{2}+cd+{d}^{2} \right) ^{18}
 \left( {b}^{2}+bd+{d}^{2} \right) ^{18}
 \left( {b}^{2}+bc+{c}^{2} \right) ^{18}
\end{smallmatrix}\right].
 \]
Using the expressions in \cite[Eq.~(5)]{BremnerHuOeding} we find that this restriction is
\begin{equation}\label{eq:proj_e8}
\left[{ \small \Res\left(\mfrac{1} {z_1^{6}}(\Res(\Delta_{\text{E}_{8}},\mathfrak{s}_1)), \mathfrak{s}'\right) } \right] =\left[ \Res(\Delta_{333}^{2},\mathfrak{s}').\Res(f_{9}^{18}, \mathfrak{s}')\right],
\end{equation}
where $f_9$ is the degree $9$ Strassen invariant.
Restricting the 72 roots of $\text{E}_{6}$ to the semi-simple elements $\mathfrak{s}'$ by the same method, we find:
\begin{equation}\label{eq:e6s333}
\left[\Res(\Delta_{\text{E}_{6} },\mathfrak{s}')\right] = \left[\Res(\Delta_{333}^{2},\mathfrak{s}')\right].
\end{equation}
By Proposition~\ref{prop:up} Eq.~\eqref{eq:e6s333} is true for all elements and therefore  provides a more precise statement than the division relation obtained in Example~\ref{ex:e6}.
\subsection{Restricting $\Delta_{\text{E}_7}$ to semi-simple elements of $\bigwedge^4 \CC^8$}
Similarly,  we use the discriminant $\Delta_{\text{E}_7}$ expressed on semi-simple elements as the product of the roots of $\mathfrak{e}_7$.
 We follow the notation in \cite{Antonyan}.
Let $\{ \alpha \in R\}$ be the usual root system for $\text{E}_7$ expressed in the hyperplane $\sum_{i}\varepsilon_{i}=0$ where $\varepsilon_{1},\ldots, \varepsilon_{8}$ is the basis of the diagonal $8\times 8$ matrices.
A basic set of semi-simple elements  $\mathfrak{s}\subset \bw{4}\CC^8$ is
\begin{equation}\label{eq:ssb48}
\begin{matrix}
p_{1}=e_{1234}+e_{5678} ,&
p_{2}=e_{1357}+e_{6824}, &
p_{3}=e_{1562}+e_{8437},&
p_{4}=e_{1683}+e_{4752},\\
p_{5}=e_{1845}+e_{7263},&
p_{6}=e_{1476}+e_{2385},&
p_{7}=e_{1728}+e_{3546},
\end{matrix} \end{equation}
where $e_{ijkl} = e_{i}\wedge e_{j}\wedge e_{k}\wedge e_{l}$.
An open set of semi-simple elements is given by the generic form
$p = \sum_{i}y_{i}p_{i},
$ for parameters $y_{i}$ not all zero.
We used the following substitutions  \cite{Antonyan}:
{\small
\medmuskip=0mu
\thinmuskip=2mu
\thickmuskip=2mu
\nulldelimiterspace=2pt
\scriptspace=2pt
\begin{align*}
\varepsilon _1(p) &= \phantom{-}y_1+y_2+y_3+y_4+y_5+y_6+y_7, &
\varepsilon_2(p) &= \phantom{-}y_1-y_2+y_3-y_4-y_5-y_6+y_7,\\[-.5ex]
\varepsilon_3(p) &= \phantom{-}y_1+y_2-y_3+y_4-y_5-y_6-y_7, &
\varepsilon_4(p) &= \phantom{-}y_1-y_2-y_3-y_4+y_5+y_6-y_7,\\[-.5ex]
\varepsilon_5(p) &= -y_1+y_2+y_3-y_4+y_5-y_6-y_7, &
\varepsilon_6(p) &= -y_1-y_2+y_3+y_4-y_5+y_6-y_7,\\[-.5ex]
\varepsilon_7(p) &= -y_1+y_2-y_3-y_4-y_5+y_6+y_7, &
\varepsilon_8(p) &= -y_1-y_2-y_3+y_4+y_5-y_6+y_7.
\end{align*}
}
Restricting $\Delta_{\text{E}_7}$ to our choice of semi-simple tensors gives $[\Res(\Delta_{\text{E}_7},\mathfrak{s})] = [h^{2}]$, with $h=$
{\begingroup
\makeatletter
\renewcommand{\maketag@@@}[1]{\hbox{\m@th\normalsize\normalfont#1}}%
\makeatother
\scriptsize
\medmuskip=-1mu
\thinmuskip=-1mu
\thickmuskip=-1mu
\nulldelimiterspace=-1pt
\scriptspace=0pt
\begin{align}&
(y_1+y_2+y_3+y_6)
(y_1+y_2+y_3-y_6)
(y_1+y_2-y_3+y_6)
(y_1+y_2-y_3-y_6)
(y_1-y_2+y_3+y_6)
(y_1-y_2+y_3-y_6)
(y_1-y_2-y_3+y_6)
(y_1-y_2-y_3-y_6)
\nonumber\\[-.5ex]&
(y_1+y_3+y_4+y_5)
(y_1+y_3+y_4-y_5)
(y_1+y_3-y_4+y_5)
(y_1+y_3-y_4-y_5)
(y_1-y_3+y_4+y_5)
(y_1-y_3+y_4-y_5)
(y_1-y_3-y_4+y_5)
(y_1-y_3-y_4-y_5)
\nonumber\\[-.5ex]&
(y_1+y_2+y_5+y_7)
(y_1+y_2+y_5-y_7)
(y_1+y_2-y_5+y_7)
(y_1+y_2-y_5-y_7)
(y_1-y_2+y_5+y_7)
(y_1-y_2+y_5-y_7)
(y_1-y_2-y_5+y_7)
(y_1-y_2-y_5-y_7)
\nonumber\\[-.5ex]&
(y_1+y_4+y_6+y_7)
(y_1+y_4+y_6-y_7)
(y_1+y_4-y_6+y_7)
(y_1+y_4-y_6-y_7)
(y_1-y_4+y_6+y_7)
(y_1-y_4+y_6-y_7)
(y_1-y_4-y_6+y_7)
(y_1-y_4-y_6-y_7)
\nonumber\\[-.5ex]&
(y_2+y_3+y_4+y_7)
(y_2+y_3+y_4-y_7)
(y_2+y_3-y_4+y_7)
(y_2+y_3-y_4-y_7)
(y_2-y_3+y_4+y_7)
(y_2-y_3+y_4-y_7)
(y_2-y_3-y_4+y_7)
(y_2-y_3-y_4-y_7)
\nonumber\\[-.5ex]&
(y_2+y_4+y_5+y_6)
(y_2+y_4+y_5-y_6)
(y_2+y_4-y_5+y_6)
(y_2+y_4-y_5-y_6)
(y_2-y_4+y_5+y_6)
(y_2-y_4+y_5-y_6)
(y_2-y_4-y_5+y_6)
(y_2-y_4-y_5-y_6)
\nonumber\\[-.5ex]&
(y_3+y_5+y_6+y_7)
(y_3+y_5+y_6-y_7)
(y_3+y_5-y_6+y_7)
(y_3+y_5-y_6-y_7)
(y_3-y_5+y_6+y_7)
(y_3-y_5+y_6-y_7)
(y_3-y_5-y_6+y_7)
(y_3-y_5-y_6-y_7)
\nonumber\\[-.5ex]&
y_1
y_2
y_3
y_4
y_5
y_6
y_7.
 \label{eq:dg48}\raisetag{20pt}
\end{align}
\endgroup}
Since $[\Delta_{\Gr(4,8)}]=[\Res(\Delta_{\text{E}_7},\bigwedge^4 \CC^8)]$ (see Eq.~\eqref{eq:e7}), the previous expression already furnishes an expression of $\Delta_ {\text{Gr}(4,8)}$ on semi-simple elements.  This will be used in Section~\ref{sec:fundamental} to obtain $\Delta_ {\text{Gr}(4,8)}$ in terms of  fundamental invariants.
\subsubsection{Restricting $\Delta_{\text{E}_7}$ to semi-simple elements of $\CC^2\otimes \CC^2\otimes \CC^2\otimes \CC^2$}\label{sec:2222} 
Now choose an ordered basis of $\CC^{8}$ in four pairs compatible with the choice of $p_{1}$.
\[
\langle e_{1}, e_{5} \rangle  \oplus
\langle e_{2}, e_{6} \rangle  \oplus
\langle e_{3}, e_{7} \rangle  \oplus
\langle e_{4}, e_{8} \rangle  = \CC^{8}
.\]
Consider the copy of $\CC^{2}\otimes \CC^{2}\otimes \CC^{2}\otimes \CC^{2}$ inside $\bw{4}\CC^{8}$ induced by this choice, and take $\{x_{I} \mid I\in \{0,1\}^{4}\} $ to be the corresponding basis of the tensor space, where $x_{I} = x_{i_{1}}\otimes x_{i_{2}}\otimes x_{i_{3}}\otimes x_{i_{4}}$ and
a 0 (respectively a 1) in position $j$ corresponds to the first (respectively the second) basis vector in factor $j$.
This induces a projection of basic semi-simple elements \eqref{eq:ssb48} in $\bw{4}\CC^{8}$ to:
\[\begin{matrix}
q_{1} = x_{0000}+x_{1111}, \quad &
q_{2} = x_{0101}+x_{1010}, \quad &
q_{3} = x_{0110}+x_{1001}, \quad &
q_{4} = x_{0011}+x_{0011},
\end{matrix}
\]
whose span we denote by $\mathfrak{s}'$.
Specifically, the coordinate projection $\mathfrak{s}\to \mathfrak{s}'$ is given by
\begin{equation}\label{eq:proj48}
p_{1}=q_{1}, \quad\quad p_{2}=0,\quad\quad  p_{3}=0,\quad\quad  p_{4} = q_{2},\quad\quad   p_{5}=0,\quad\quad  p_{6} = q_{3},\quad \quad p_{7} = q_{4}.
 \end{equation}
Write a generic semi-simple form in $\mathfrak{s}'$ as $t_{1}q_{1} + t_{2} q_{2} + t_{3} q_{3} + t_{4} q_{4}$ for parameters $t_{i}$ not all zero.
On $\mathfrak{s}'$ the  $2\times 2\times 2\times 2$ hyperdeterminant is the Vandermonde determinant squared \cite{holweck_4qubit2}:\looseness-1
\[
\Res(\Delta_{2222}, \mathfrak s')= \det
\left( \begin{smallmatrix} 1&1&1&1\\[.5ex]
{t_{1}}^{2}&{t_{2}}^{2}&{t_{3}}^{2}&{t_{4}}^{2}\\[.5ex]
{t_{{1}}}^{4}&{t_{2}}^{4}&{t_{3}}^{4}&{t_{4}}^{4}\\[.5ex]
{t_{1}}^{6}&{t_{2}}^{6}&{t_{3}}^{6}&{t_{4}}^{6}\end{smallmatrix}
\right)^{\hspace{-1ex} 2} = \prod_{i<  j} ((t_{i}-t_{j})(t_{i}+t_{j}))^2.
 \]
The roots of $\text{E}_{7}$ contain 6 roots that vanish on $\mathfrak{s}'$. This is evident from \eqref{eq:dg48}, where the restriction of $\Res(\Delta_{\Gr(4,8)},\mathfrak s)$ contains $(y_{2}y_{3}y_{5})^{2}$, which we set to zero in \eqref{eq:proj48}.
When those 6 factors are cancelled we obtain a polynomial of degree 120 that has the 4-th power of this restriction of the hyperdeterminant as a factor:
{
\makeatletter
\renewcommand{\maketag@@@}[1]{\hbox{\m@th\normalsize\normalfont#1}}%
\makeatother
\fontsize{8}{9}\selectfont
\medmuskip=-.5mu
\thinmuskip=-.5mu
\thickmuskip=0mu
\nulldelimiterspace=0pt
\scriptspace=-0.75pt
\begin{align*}
\Res\left(\frac{\Res(\Delta_{ \text{E}_{7}}\hspace{.2em} ,\hspace{.2em}  \mathfrak s)  }{(y_{2}y_{3}y_{5})^{2}}\hspace{.2em} ,\hspace{.2em} \mathfrak{s}'\right) =\;&
(t_{1}+t_{2})^{8}
(t_{1}-t_{2})^{8}
(t_{1}+t_{3})^{8}
(t_{1}-t_{3})^{8}
(t_{1}+t_{4})^{8}
(t_{1}-t_{4})^{8}
(t_{2}+t_{3})^{8}
(t_{2}-t_{3})^{8}
(t_{2}+t_{4})^{8}
(t_{2}-t_{4})^{8}
(t_{3}+t_{4})^{8}
(t_{3}-t_{4})^{8}
\nonumber \\[-2.25ex]& \cdot
(t_{1}+t_{2}+t_{3}+t_{4})^{2}
(t_{1}-t_{2}-t_{3}+t_{4})^{2}
(t_{1}+t_{2}-t_{3}-t_{4})^{2}
(t_{1}-t_{2}+t_{3}-t_{4})^{2}
(t_{1}+t_{2}-t_{3}+t_{4})^{2}
(t_{1}-t_{2}+t_{3}+t_{4})^{2}
\nonumber \\[-.5ex]&\cdot
(t_{1}+t_{2}+t_{3}-t_{4})^{2}
(t_{1}-t_{2}-t_{3}-t_{4})^{2}
{t_{1}}^{2}{t_{2}}^{2}{t_{3}}^{2}{t_{4}}^{2}.
\end{align*}
}

\section{The $\Gr(3,9)$ and $\Gr(4,8)$ discriminants in fundamental invariants}\label{sec:fundamental}
We use our approach to write the $\Gr(3,9)$- and $\Gr(4,8)$-discriminants in terms of fundamental invariants of lower degree. This has been done before in other settings,
see \cite{Morozov2010, BremnerHuOeding, CDZG}.
\subsection{$\Gr(3,9)$}
The invariant ring $\CC[\bw{3}\CC^{9}]^{\SL(9)}$  is a polynomial ring generated by fundamental invariants denoted $f_{i}$ for degrees $i=12, 18, 24, 30$. These invariants can be computed explicitly as traces of powers of Katanova's matrix \cite{Katanova}, which we recall in Section~\ref{sec:gr39gr48}. We can also restrict those invariants to semi-simple elements $\mathfrak{s}\subset \bw{3}{\CC^9}$ and express them in terms of the roots of $\text{E}_8$ using the following isomorphisms induced by the restriction maps:
\[
 \CC[\mathfrak{e}_8]^{\text{E}_8}\isom \CC[\mathfrak{h}]^{W_{\text{E}_8}}\to  \CC[\mathfrak{s}]^{W_{SL_9}} \isom \CC[\bw{3}\CC^9]^{\SL(9)} ,
\]
where $W_G$ denotes the Weyl group of $G$.
In this case because  $W_{\text{E}_8}$ acts transitively on the root system, invariants of degree $i$ of $\CC[\mathfrak{h}]^{W_{\text{E}_8}}$ can be chosen to be the sum of the $i$-th power of the roots and therefore 
\begin{equation}\label{eq:fundinvariants}
\Res(f_{i},\mathfrak{s})=\sum_{\alpha\in R} \Res(\alpha ^i,\mathfrak{s}).
\end{equation}
These expressions agree with Katanova's up to a swap between $z_{2}$ and $z_{3}$, which is a simple base-change.
The terms occurring in the fundamental invariants are symmetric under a signed permutation action, which is generated by 
$\alpha = (z_{1}\rightarrow z_{2},z_{2}\rightarrow -z_{3},z_{3}\rightarrow z_{4}, z_{4}\rightarrow z_{1} )$ and $\beta=  (z_{1}\rightarrow z_{2},z_{2}\rightarrow z_{1},z_{3}\rightarrow -z_{3}, z_{4}\rightarrow z_{4} )$. Figure~\ref{fig:f39} gives a list of the terms occurring in each invariant up to the symmetry group $\langle \alpha, \beta\rangle$.

\begin{figure}[!b]
{\Small \begin{align*}
f_{12}: &\;
z_{1}^{12},\; 22 z_{1}^{6} z_{2}^{6},\; -220 z_{1}^{6} z_{2}^{3} z_{3}^{3}
\\[1ex]
f_{18}:&\;
z_{1}^{18},\; {-17 z_{1}^{12} z_{2}^{6}},\; 170 z_{1}^{12} z_{2}^{3} z_{3}^{3},\; 1870 z_{1}^{9} z_{2}^{6} z_{3}^{3},\; -7854 z_{1}^{6} z_{2}^{6} z_{3}^{6}
\\[1ex]
f_{24}: &\;
111 z_{1}^{24},\; 506 z_{1}^{18} z_{2}^{6},\; 10166 z_{1}^{12} z_{2}^{12},\; {-5060 z_{1}^{18} z_{2}^{3} z_{3}^{3}},\; {-206448 z_{1}^{15} z_{2}^{6} z_{3}^{3}},\; {-1118260 z_{1}^{12} z_{2}^{9} z_{3}^{3}},
\\[-.5ex]&\; 4696692 z_{1}^{12} z_{2}^{6} z_{3}^{6},\; 12300860 z_{1}^{9} z_{2}^{9} z_{3}^{6}
\\[1ex]
f_{30}: &\;
584 z_{1}^{30},\; {-435 z_{1}^{24} z_{2}^{6}},\; {-63365 z_{1}^{18} z_{2}^{12}},\; 4350 z_{1}^{24} z_{2}^{3} z_{3}^{3},\; 440220 z_{1}^{21} z_{2}^{6} z_{3}^{3},\; 6970150 z_{1}^{18} z_{2}^{9} z_{3}^{3},\; 25852920 z_{1}^{15} z_{2}^{12} z_{3}^{3},
\\[-.5ex]&\; {-29274630 z_{1}^{18} z_{2}^{6} z_{3}^{6}},\; {-284382120 z_{1}^{15} z_{2}^{9} z_{3}^{6}},\; {-588153930 z_{1}^{12} z_{2}^{12} z_{3}^{6}},\; 1540403150 z_{1}^{12} z_{2}^{9} z_{3}^{9}
\end{align*}}
\caption{Terms up to symmetry by $\langle\alpha, \beta \rangle$ occurring in expressions of fundamental invariants for $\CC[\bw{3}\CC^{9}]^{\SL(9)}$ restricted to semi-simple parts.}\label{fig:f39}
\end{figure}

\begin{theorem}
 Let us consider the invariant ring  $\CC[\bigwedge^3 \CC^9]^{\SL(9)}=\CC[f_{12},f_{18},f_{24},f_{30}]$. Then the $\Gr(3,9)$-discriminant, $\Delta_{\Gr(3,9)}$, is a degree $120$ polynomial that can be expressed in terms of the fundamental invariants as:
 $\Delta_{\Gr(3,9)}=$
{\begingroup
\makeatletter
\renewcommand{\maketag@@@}[1]{\hbox{\m@th\normalsize\normalfont#1}}%
\makeatother
\footnotesize
\medmuskip=-.5mu
\thinmuskip=-1mu
\thickmuskip=-1mu
\scriptspace=0pt
\begin{align}
& f_{12}^{10}-
\mfrac{188875}{1526823}\,f_{12}^{8}f_{24}-
\mfrac{44940218765172270463}{2232199994248855116}\,f_{12}^{7}f_{18}^{2}+
\mfrac{522717082571600510}{5022449987059924011}\,f_{12}^{6}f_{18}f_{30}+
\mfrac{156259946875}{27974261679948}\,f_{12}^{6}f_{24}^{2}
\nonumber\\ &
+
\mfrac{20955843759677134000}{15067349961179772033}\,f_{12}^{5}f_{18}^{2}f_{24}+
\mfrac{113325967730636958495085217}{1009180965699898771226274}\,f_{12}^{4}f_{18}^{4}-
\mfrac{8007699664851700}{45202049883539316099}\,f_{12}^{5}f_{30}^{2}
\nonumber\\ &
-
\mfrac{951594557840795000}{135606149650617948297}\,f_{12}^{4}f_{18}f_{24}f_{30}-
\mfrac{37339826093750}{327991224631970313}\,f_{12}^{4}f_{24}^{3}-
\mfrac{4631798176278228432974860}{4541314345649544470518233}\,f_{12}^{3}f_{18}^{3}f_{30}
\nonumber\\ &
-
\mfrac{43381098724294271875}{2440910693711123069346}\,f_{12}^{3}f_{18}^{2}f_{24}^{2}-
\mfrac{48098757899275092625}{15067349961179772033}\,f_{12}^{2}f_{18}^{4}f_{24}-
\mfrac{11518845901768651039}{329340982758027804}\,f_{12}f_{18}^{6}
\nonumber\\ &
+
\mfrac{1392403335812500}{135606149650617948297}\,f_{12}^{3}f_{24}f_{30}^{2}+
\mfrac{6686357462527147925300}{1513771448549848156839411}\,f_{12}^{2}f_{18}^{2}f_{30}^{2}+
\mfrac{140973248590625000}{1220455346855561534673}\,f_{12}^{2}f_{18}f_{24}^{2}f_{30}
\nonumber\\ &
+
\mfrac{351718750000}{327991224631970313}\,f_{12}^{2}f_{24}^{4}+
\mfrac{2133816827644645000}{135606149650617948297}\,f_{12}f_{18}^{3}f_{24}f_{30}-
\mfrac{198339133437500}{741017211205562559}\,f_{12}f_{18}^{2}f_{24}^{3}
\nonumber\\ &
+
\mfrac{45691574382263590}{741017211205562559}\,f_{18}^{5}f_{30}-
\mfrac{32778366465625}{48591292538069676}\,f_{18}^{4}f_{24}^{2}-
\mfrac{14445540571041712000}{1513771448549848156839411}\,f_{12}f_{18}f_{30}^{3}
\nonumber\\ &
-
\mfrac{216716472500000}{1220455346855561534673}\,f_{12}f_{24}^{2}f_{30}^{2}-
\mfrac{2371961791512500}{135606149650617948297}\,f_{18}^{2}f_{24}f_{30}^{2}+
\mfrac{10890275000000}{20007464702550189093}\,f_{18}f_{24}^{3}f_{30}
\nonumber\\ &
-
\mfrac{1250000000}{327991224631970313}\,f_{24}^{5}+
\mfrac{34328756109890000}{4541314345649544470518233}\,f_{30}^{4}.
\end{align}
\endgroup}\label{eq:g39}
\hspace{-1ex}Moreover $\Delta_{\Gr(3,9)}$ has the geometric interpretation of being the restriction of $\Delta_{\text{E}_8}$ on $\bigwedge^3 \CC^9$. More precisely we have $[\Delta_{\Gr(3,9)}^2]=[\Res(\Delta_{\text{E}_8},\bigwedge^3 \CC^9)]$.
\end{theorem}

\begin{proof}
This proof relies on computations, so we provide the scripts needed to reproduce our results as well as the outputs in the ancillary files accompanying the arXiv version of this paper.
The discriminant $\Delta_{\Gr(3,9)} \in \CC[\bw{3}\CC^{9}]^{\SL(9)}$ has degree 120 \cite{LascouxDegreeDualGrassmannian}. Since the invariant ring is a polynomial ring, we can express $\Delta_{\Gr(3,9)}$ as a polynomial in the fundamental invariants via standard linear interpolation in rational arithmetic. We can either do this computation directly in terms of the $f_i$'s or by first finding the relation between $\Res(\Delta_{\Gr(3,9)},\mathfrak{s})$ and the invariants $\Res(f_{i},\mathfrak{s})$ and then lifting the resulting expression using Proposition~\ref{prop:up}. 
 To be more precise, there are 28 monomials in $f_{12}, f_{18}, f_{24}, f_{30}$ of degree $120$. These monomials are expressed in terms of semi-simple elements up to symmetry in Figure~\ref{fig:f39}. To obtain the coefficients of the degree $120$ invariant corresponding to $\Delta_{\Gr(3,9)}$ we evaluate at sufficiently many points of $\Gr(3,9)^{\vee}$ and solve a linear system. We restrict our calculation to the set of semi-simple elements, i.e., to a generic element of $\mathfrak{s}\subset \bigwedge^3 \CC^9$ (See Section~\ref{sec:gr39gr48}):
 \begin{equation}\label{eq:generic_tensor}
  p=z_1p_1+z_2p_2+z_3p_3+z_4p_4.
 \end{equation}
A straightforward calculation shows that whenever one of the $z_i$ is zero then the hyperplane $H_p=\{v\in \bigwedge^3 \CC^9, p^*(v)=0\}$ belongs to $\Gr(3,9)^\vee$.  By randomly choosing three non-zero variables in Eq.~\eqref{eq:generic_tensor} one can get enough independent relations to determine the $28$ coefficients of $\Delta_{\Gr(3,9)}$. 
The ancillary Maple file \path{gr39_dual_1.mw} performs this calculation.

Using the expression \eqref{eq:g39} on semi-simple elements and collecting terms we find
that up to a non-zero constant $\Res(\Delta_{\Gr(3,9)},\mathfrak{s})$ has the form $h^{3}$ with $h = $
%
{\begingroup
\makeatletter
\renewcommand{\maketag@@@}[1]{\hbox{\m@th\normalsize\normalfont#1}}%
\makeatother
\fontsize{8}{9}\selectfont
\medmuskip=-.5mu
\thinmuskip=-.5mu
\thickmuskip=0mu
\nulldelimiterspace=0pt
\scriptspace=-0.75pt
\begin{align}
& z_{1}z_{2}z_{3}z_{4}(z_{1}+z_{2}-z_{3})( z_{1}+z_{3}-z_{4})( z_{1}-z_{2}+z_{4})( z_{2}+z_{3}+z_{4}) \nonumber \\[-.5ex]&\cdot
( z_{1}^2-z_{1}z_{2}+z_{2}^2+z_{1}z_{3}-2z_{2}z_{3}+z_{3}^2) 
( z_{1}^2+z_{1}z_{2}+z_{2}^2-z_{1}z_{4}-2z_{2}z_{4}+z_{4}^2)
(z_{1}^2-z_{1}z_{3}+z_{3}^2+z_{1}z_{4}-2z_{3}z_{4}+z_{4}^2)
( z_{2}^2-z_{2}z_{3}+z_{3}^2-z_{2}z_{4}+2z_{3}z_{4}+z_{4}^2)\nonumber \\[-.5ex]& \cdot
( z_{1}^2-2z_{1}z_{2}+z_{2}^2-z_{1}z_{4}+z_{2}z_{4}+z_{4}^2)
(z_{1}^2-z_{1}z_{2}+z_{2}^2-2z_{1}z_{3}+z_{2}z_{3}+z_{3}^2)
( z_{1}^2-z_{1}z_{2}+z_{2}^2+z_{1}z_{3}+z_{2}z_{3}+z_{3}^2)
( z_{1}^2+z_{1}z_{2}+z_{2}^2-z_{1}z_{4}+z_{2}z_{4}+z_{4}^2)\nonumber \\[-.5ex]&\cdot
(z_{1}^2+z_{1}z_{2}+z_{2}^2+2z_{1}z_{4}+z_{2}z_{4}+z_{4}^2)
( z_{1}^2+2z_{1}z_{2}+z_{2}^2+z_{1}z_{3}+z_{2}z_{3}+z_{3}^2)
( z_{1}^2-z_{1}z_{3}+z_{3}^2-2z_{1}z_{4}+z_{3}z_{4}+z_{4}^2)
(z_{1}^2-z_{1}z_{3}+z_{3}^2+z_{1}z_{4}+z_{3}z_{4}+z_{4}^2)\nonumber \\[-.5ex]&\cdot
( z_{1}^2+2z_{1}z_{3}+z_{3}^2+z_{1}z_{4}+z_{3}z_{4}+z_{4}^2)
( z_{2}^2-z_{2}z_{3}+z_{3}^2-z_{2}z_{4}-z_{3}z_{4}+z_{4}^2)
(z_{2}^2-z_{2}z_{3}+z_{3}^2+2z_{2}z_{4}-z_{3}z_{4}+z_{4}^2)
( z_{2}^2+2z_{2}z_{3}+z_{3}^2-z_{2}z_{4}-z_{3}z_{4}+z_{4}^2).
\label{eq:del39}
\end{align}
\endgroup}
Comparing \eqref{eq:del39} with  \eqref{eq:gr39z} we obtain the relation $[\Delta_{\Gr(3,9)}^2] = [\Res(\Delta_{\text{E}_8},\bigwedge^3 \CC^9)]$ on semi-simple elements. Proposition~\ref{prop:up} establishes this relation for all elements. This shows that how the dual variety $\Gr(3,9)^{\vee}$ is obtained as a projection of the $\text{E}_8$-discriminant and provides a more precise statement than the first division relation of Example~\ref{ex:e8}. \qedhere
\end{proof}

\begin{remark}
 Using the expression of the projection of $\Delta_{\text{E}_8}$ to $(\CC^3)^{\otimes 3}$ given by Eq.~\eqref{eq:proj_e8} and the relation $[\Delta_{\Gr(3,9)}^2] = [\Res(\Delta_{\text{E}_8},\bigwedge^3 \CC^9)]$, one can consider the restriction of $\Delta_{\Gr(3,9)}$ to $(\CC^{3})^{\otimes 3}$ to semi-simple elements. Comparing \eqref{eq:del39} and \cite[Eq.~(5)]{BremnerHuOeding} we see that
\[
\left[ \Res\left(\mfrac{1} {a^{3}}\Delta_{\Gr(3,9)},\mathfrak{s}'\right) \right]
=
\left[ \Res(\Delta_{333},\mathfrak{s}')\right].
\]
\end{remark}

\subsection{$\Gr(4,8)$}
A similar analysis can be done for $\bw{4}\CC^{8}$. Again, the invariant ring is a polynomial ring generated now by fundamental invariants of degrees 2, 6, 8, 10, 12, 14, 18.
By Proposition~\ref{prop:up} it suffices to work on a generic semi-simple element. A fundamental invariant in $\CC[\bw{4}\CC^{8}]^{\SL(8)}$ of degree $d$ can be computed on semi-simple tensors as $d$-th powers of sums of the roots of $\text{E}_{7}$ (see \eqref{eq:fundinvariants}). These polynomials are symmetric in the variables $y_{1}\ldots y_{7}$.
Figure~\ref{fig:f8} gives a list of their terms up to symmetry. Note  the coefficients are highly composite.
\begin{figure}
{\Small \begin{align*}
f_{2}: &\; y_{1}^{2}
\\[1ex]
f_6:&\;
2y_1^6,\;5y_1^4y_2^2,\;15y_1^2y_2^2y_3^2
\\[1ex]
f_8: &\;
9y_1^8,\;14y_1^6y_2^2,\;35y_1^4y_2^4,\;105y_1^4y_2^2y_3^2,\;630y_1^2y_2^2y_3^2y_4^2
\\[1ex]
f_{10}:&\;
22y_1^{10},\;15y_1^8y_2^2,\;70y_1^6y_2^4,\;210y_1^6y_2^2y_3^2,\;525y_1^4y_2^4y_3^2,\;3150y_1^4y_2^2y_3^2y_4^2
\\[1ex]
f_{12}:&\;
86y_1^{12},\;22y_1^{10}y_2^2,\;165y_1^8y_2^4,\;308y_1^6y_2^6,\;495y_1^8y_2^2y_3^2,\;2310y_1^6y_2^4y_3^2,\;5775y_1^4y_2^4y_3^4,
 \;13860y_1^6y_2^2y_3^2y_4^2,\;34650y_1^4y_2^4y_3^2y_4^2
\\[1ex]
f_{14}:
&\;
2052y_1^{14},\;182y_1^{12}y_2^2,\;2002y_1^{10}y_2^4,\;6006y_1^8y_2^6,\;6006y_1^{10}y_2^2y_3^2,\;45045y_1^8y_2^4y_3^2,\;84084y_1^6y_2^6y_3^2,\;210210y_1^6y_2^4y_3^4,
\\[-.5ex] &\;270270y_1^8y_2^2y_3^2y_4^2,\;1261260y_1^6y_2^4y_3^2y_4^2,\;3153150y_1^4y_2^4y_3^4y_4^2
\\[1ex]
f_{18}:
&\;
5462y_1^{18},\;51y_1^{16}y_2^2,\;1020y_1^{14}y_2^4,\;6188y_1^{12}y_2^6,\;14586y_1^{10}y_2^8,\;3060y_1^{14}y_2^2y_3^2,\;46410y_1^{12}y_2^4y_3^2,\;204204y_1^{10}y_2^6y_3^2,
\\[-.5ex]&
\;328185y_1^8y_2^8y_3^2,\;510510y_1^{10}y_2^4y_3^4,\;1531530y_1^8y_2^6y_3^4,\;2858856y_1^6y_2^6y_3^6,\;278460y_1^{12}y_2^2y_3^2y_4^2,\;3063060y_1^{10}y_2^4y_3^2y_4^2,
\\[-.5ex]&
\;9189180y_1^8y_2^6y_3^2y_4^2,\;22972950y_1^8y_2^4y_3^4y_4^2,\;42882840y_1^6y_2^6y_3^4y_4^2,\;107207100y_1^6y_2^4y_3^4y_4^4
\end{align*}}
\caption{Terms up to symmetry occurring in expressions of fundamental invariants for $\CC[\bw{4}\CC^{8}]^{\SL(8)}$ restricted to semi-simple parts.}\label{fig:f8}
\end{figure}

\begin{theorem}
Consider the invariant ring $\CC[\bigwedge^4 \CC^8]^{\SL(8)}=\CC[f_2,f_6,f_8,f_{10},f_{12},f_{18}]$. Then the $\Gr(4,8)$-discriminant, $\Delta_{\Gr(4,8)}$ is a degree $126$ polynomial that can be expressed in terms of $15,942$ monomials of fundamental invariants:

  {\begingroup
\makeatletter
\renewcommand{\maketag@@@}[1]{\hbox{\m@th\normalsize\normalfont#1}}%
\makeatother
\fontsize{8}{9}\selectfont
 \begin{align}\Delta_{\Gr(4,8)} =\;&
-(11228550634163820692582736367065066800237662227759449345598
  861374381270810701586235392
  \nonumber \\[-.5em]
  &/1900359976262346454474448419809074
  880484088763429831167939681466204604687770731158447265625)
      f_{2}^{63}  \nonumber \\[-.5em]
      &
      + \cdots +(3/1690514664168754070821429178618909) f_{18}^{7},
\label{eq:del48}
\end{align}
\endgroup}
\hspace{-1ex}\noindent
Moreover $\Delta_{\Gr(4,8)}$ has the geometric interpretation of being the projection of $\Delta_{\text{E}_7}$ to $\bigwedge ^4 \CC^8$. More precisely we have 
$[\Delta_{\Gr(4,8)}]=[\Res(\Delta_{\text{E}_7},\bigwedge^4 \CC^8)]$.
\end{theorem}
\begin{proof} The ``moreover'' part of the Theorem was proved in Eq.(\ref{eq:e7}) and we obtained in Section~\ref{sec:semi-simple} an expression of the $\Gr(4,8)$ discriminant on semi-simple elements. The dual $\Gr(4,8)^\vee$ has degree $126$ \cite{LascouxDegreeDualGrassmannian} and there are $15,976$ monomials in the fundamental invariants with degree $126$.
In theory, linear interpolation over the rationals should work as follows. Evaluating each monomial in fundamental invariants and the discriminant at at least $n=15,976+1$ points (in practice we use 10\% more points in addition to help ensure that the points we choose at random are not in special position). Store these results in the rows of a matrix and compute a basis of its null space.
Though we were able to generate a $\lfloor1.1n \rfloor \times n$ matrix by this method, it was dense and had large integer entries (approximately 6GB of disk space space), and we were not able to finish the null space computation in rational arithmetic due to memory issues caused by coefficient explosion in the intermediate results.

Working modulo a prime $p$ we can avoid coefficient growth and complete the interpolation problem, use Proposition~\ref{prop:up} and express $\Delta_{\Gr(4,8)}$ as a polynomial in the fundamental invariants modulo $p$. Using a server with 32 processors, for small primes, this computation took approximately 6 hours and the resulting expression for $\Delta_{\Gr(4,8)}$ can be stored as a list of coefficients of size approximately 78kb. These are the ancillary files \path{gr48/503} through \path{gr48/9533}.  The ancillary Maple file \path{gr48/input.mw} performs this computation.

We repeated this computation on Auburn's CASIC cluster for over 100 different instances for the first 100 primes above 1000  to produce 100 reductions of the true null vector (each of these computations took approximately 4 hours, but we ran them in parallel). We then used the Chinese remainder theorem in each place to produce an integer vector modulo $N$, with $N$ equal to the product of all but one of the primes. Then we used rational reconstruction in each coordinate to find each unique rational number equivalent to each of our coordinates modulo $N$. The Maple script to reproduce this lifting is at \path{gr48/lifting.mw}.
The final result is a polynomial in the fundamental invariants involving  15,942 terms which can be found in the ancillary file \path{gr48/hd48_100new.txt}.

We checked that the coefficient vector reduces to the same result modulo several other primes. We also checked that for new random points on the discriminant locus that this expression for the discriminant vanishes (without reducing modulo $p$). These probabilistic checks are not a proof. We  certify the rational reconstruction  solution via the following:
\begin{lemma}{{\cite[Lemma~2.4]{steffy2011exact}}}\label{lem:cert}
Suppose $A$ is an $n\times n$  integer matrix, $y$, $b$, are integer vectors, and $d\geq 0$ is an integer. If for some integer $M$
\[
Ay = bd\quad \mod  M \quad\quad\text{and}
\quad \max(d || b||_{\infty}, n || A ||_{\max}||y||_{\infty}) <M/2,
\]
then $Ay = bd$.
\end{lemma}
In particular, if $d$ is the zero vector, then $y$ is a solution to $Ax=0$ if $n || A ||_{\max}||y||_{\infty} <M/2$.
In our case the maximum entry in $A$ is less than $10^{200}$. The lcm of the denominators in $y$ is $<10^{115}$, and $n=15976 <10^{6}$.
The product of the first 100 primes with 4 digits gives $M = 0.2\ldots \times 10^{360}$. Since
\[
 n || A ||_{\max}||y||_{\infty}  < 10^{321}<M/2 = 0.1\ldots \times 10^{360},
\]
Lemma~\ref{lem:cert} applies, and we have a certified result. 
\end{proof}

\begin{remark}
Recall (\ref{sec:2222}) that the choice of a copy of $\CC^2\otimes\CC^2\otimes\CC^2\otimes\CC^2$ in $\bw{4}\CC^8$ induces a projection $\mathfrak{s}\to \mathfrak{s}'$ between semi-simple elements of $\bw{4}\CC^8$ and semi-simple elements of $\CC^2\otimes\CC^2\otimes\CC^2\otimes \CC^2$.
The invariant ring of $\SL(2)^{\times 4} \ltimes \mathfrak{S}_{4}$ is generated in degrees $2,6, 8,12$ \cite{Chterental-Dokovic}. Again, we can compute invariants of degree $d$ by taking the sums of the $d$-th powers of the roots of $\text{E}_{7}$ restricted to this abelian subalgebra $\mathfrak{s}'$. The lists of monomials up to symmetry in the invariants for $\text{E}_{7}$ on semi-simple elements contain exponent vectors of length at most 4. Since the projection is variable substitution, by setting 3 variables to zero the same expressions can be used to find the expressions of invariants for $\mathfrak{S}_{4}\ltimes \SL(2)^{\times 4}$ restricted to $\mathfrak{s}'$.
Using rational interpolation on semi-simple elements and Proposition~\ref{prop:up} to lift the resulting expression, (up to scale) the $2\times 2\times 2\times 2$ hyperdeterminant has the following expression as a polynomial in lower degree invariants:
{\begingroup
\makeatletter
\renewcommand{\maketag@@@}[1]{\hbox{\m@th\normalsize\normalfont#1}}%
\makeatother
\small \begin{align}
\Delta_{2222}=&
-\mfrac{1}{151875}f_{12}^2+\mfrac{4}{421875}f_8^3+\mfrac{496}{455625}f_6^2f_{12}-\mfrac{61504}{1366875}f_6^4
-\mfrac{88}{140625}f_2f_6f_8^2+\mfrac{922}{2278125}f_2^2f_8f_{12}
\nonumber \\&
-\mfrac{672832}{34171875}f_2^2f_6^2f_8
-\mfrac{5128}{759375}f_2^3f_6f_{12}
+\mfrac{5208736}{11390625}f_2^3f_6^3
-\mfrac{178501}{34171875}f_2^4f_8^2
+\mfrac{1865048}{11390625}f_2^5f_6f_8
\nonumber \\&
+\mfrac{3026}{455625}f_2^6f_{12}
-\mfrac{61462384}{34171875}f_2^6f_6^2-\mfrac{1156846}{6834375}f_2^8f_8+\mfrac{6012304}{2278125}f_2^9f_6
-\mfrac{1733509}{1366875}f_2^{12}.
\label{hd2222}
\raisetag{20pt}
\end{align}
\endgroup
}
\hspace{-3pt}Relations between the images of the $\SL(8)$ invariants of $\bw{4}\CC^{8}$ in $(\CC^{2})^{\otimes 4}$, which we still denote by $f_{i}$, were given in  \cite{CDZG}. Our expressions that follow, again obtained by linear interpolation, differ from  \cite{CDZG} because of different re-scalings of the fundamental invariants:
{\small\begin{align*}
f_{10} =&\mfrac{1}{5}\,f_{{2}} \left( 105\,{f_{{2}}}^{4}-119\,f_{{2}}f_{{6}}+27\,f_{{8}} \right),\\
f_{14} =& -\mfrac{6883811}{675}f_2^7+\mfrac{13205038}{1125}f_2^4f_6-\mfrac{4861087}{3375}f_2^3f_8-\mfrac{217448}{135}f_2f_6^2+\mfrac{15407}{225}f_2f_{12}+\mfrac{9548}{75}f_6f_8,\\
f_{18}=&-\mfrac{172149469}{2025}f_2^9+\mfrac{291940571}{3375}f_2^6f_6-\mfrac{99031064}{10125}f_2^5f_8-\mfrac{200876924}{50625}f_2^3f_6^2+\mfrac{45023}{135}f_2^3f_{12}
 \\ &
-\mfrac{1311188}{1875}f_2^2f_6f_8+\mfrac{42699}{625}f_2f_8^2-\mfrac{240176}{225}f_6^3+\mfrac{1204}{25}f_6f_{12}.
\end{align*}
}
\end{remark}

\section{Evaluation of the $\Gr(3,9)$ and $\Gr(4,8)$ discriminants} \label{sec:gr39gr48}
The expressions for $\Delta_{\Gr(3,9)}$ and $\Delta_{\Gr(4,8)}$ are dauntingly huge, and unknown.
However, if one has the Jordan decomposition on tensors in these spaces, one may use our formulas on semi-simple elements to evaluate the discriminants. In this section we describe how to evaluate the discriminants without appealing to Jordan decomposition.

\subsection{Warm-up with $\Gr(3,6)$} We return to Example~\ref{ex:gr36}. 
The following Katanova-type construction computes the invariants for $\Gr(3,6)$ and related varieties. Let $V \cong \CC^6$.
Consider the natural linear maps associated to a tensor $T\in \bw{3}V$ given by multiplication by $T$:
\[
\bw{1}V\to \bw{4}V\quad\text{and}\quad \bw{2} V \to \bw{5}V.
\]
Via isomorphisms induced by a volume form $\Omega:\bw{4}V\to \bw{2}V$ and $\Omega:\bw{5}V\to \bw{1}V$ we obtain the linear map:
\[
A:=\Omega T\Omega T \colon \bw{1}V \to \bw{1}V.
\]
One checks that  $\tr(A)=0$.  The invariant ring is generated in degree 4, and $f = \tr(A^{2})$, is a non-zero invariant polynomial of degree 4, hence the generator we're looking for.
Given a decomposition into a sum of 2-dimensional vector spaces $U_{1}\oplus U_{2}\oplus U_{3}= V$, we obtain a decomposition of $\bw{3}V$ which contains $U_{1}\otimes U_{2}\otimes U_{3}$. If we further identify $U_{i} = U$ for all $i$ we obtain a copy of $S^{3}U$ in $U^{\otimes 3}$.
We have the following projections to subspaces: $\bw{3} V \to U_{1}\otimes U_{2}\otimes U_{3} \to  S^{3}(U)$ which we may use to obtain the basic invariants for the Segre and Veronese cases from the Grassmann case via specialization.

\subsection{Evaluating the $\Gr(3,9)$-discriminant}
An expression for the $3\times 3\times 3$ hyperdeterminant in fundamental invariants was given in \cite{BremnerHuOeding}. Here we will describe our methods for evaluating this invariant using the methods of this article.
Katanova \cite{Katanova} gave explicit expressions of the fundamental invariants in this case. She constructs an $84\times 84$ matrix $C$ with entries that are cubic  $\bw{3}\CC^{9}$. Katanova's formula \cite[eq.~(9)]{Katanova} is for a tensor $T \in \bw{3}V$:
\[
(C(T))_{\{i\}}^{mnp} = \varepsilon_{\{ l\}\{ i\}j( k)} T^{\{l\}}(T^{m(k)}T^{jnp} +T^{n(k)}T^{mjp} + T^{p(k)}T^{mnj}  )
,\]
where $\varepsilon$ is a 9-dimensional volume form, and she uses Einstein summation convention,  denoting triple indices by $\{l\} = l_1l_2l_3$ and double indices by $(k) = k_1k_2$. 
The matrix $C(T)$ is best constructed using sparse arrays on software that is optimized for computations with such data structures, like Mathematica. However, once $C$ is constructed,  it is efficient to evaluate $C(T)$ and the powers of its traces on a particular tensor $T \in \bw{3}\CC^{9}$. Write $C$ for the matrix that is constructed in the same way using variables in place of the entries of a tensor $T$. 
The power-traces $f_{3n} = \tr(C^{n})$ are the fundamental invariants for $n=4,6,8,10$.  
\subsection{Evaluating the $\Gr(4,8)$-discriminant}
Again, Katanova gave explicit expressions for the fundamental invariants as traces of powers of a special matrix. In this case we can give a simple construction of her matrix without relying on a sum over a large index set.

A generic tensor $T \in \bw{4}\CC^{8}$ defines mappings by concatenation $\wedge T\colon \bw{2}\CC^{8}\to \bw{6}\CC^{8}$ and by contraction $\lrcorner T\colon \bw{6}\CC^{8}\to \bw{2}\CC^{8}$. Their composition is $A\colon \bw{2}\CC^{8}\to \bw{2}\CC^{8}$, a matrix whose entries are quadratic in the entries of $T$. One checks that the traces of the powers of appropriate degrees produce the fundamental invariants:
\[
f_{2n} = \tr(A^{n}),\quad \text{for} \quad n = 1, 3, 4,5,6,7,9.
\]
This can be done in a few lines in Macaulay2, for instance.
\begin{verbatim}
mysubsets = (a,b) -> apply(subsets(a,b), xx-> xx+apply(b, i-> 1));
R = QQ[e_1..e_8,SkewCommutative=> true];
S = QQ[apply(mysubsets(8,4), i-> x_i)];
RS = R**S;
myRules = sum(mysubsets(8,4),  I -> e_(I_0)*e_(I_1)*e_(I_2)*e_(I_3) *x_I);
b2 = sub(basis(2,R),RS);  b6 = sub(basis(6,R),RS);  b8 = sub(basis(8,R),RS);
A1 = diff(diff((transpose b2),b6), myRules);
A2 = diff(transpose diff(diff((transpose b2),b6),b8), myRules);
A =sub( A2*A1,S);  f_2 = trace A;   time f_6 = trace A^3;
\end{verbatim}
While the symbolic traces are unlikely to finish in a reasonable amount of time, evaluating $A$ first at a specific tensor will allow the computations to finish. For example, in Macaulay2 one may evaluate the invariants on the tensor $p_{1234}+p_{5678}$ by continuing with the following:
\begin{verbatim}
use S; myA = sub(sub(A, {x_{1,2,3,4}=>1, x_{5,6,7,8}=>1 }),QQ);
for i in {1,3,4,5,6,7,9} do  t_(2*i) = trace(myA^i);
\end{verbatim}
In our computation of  $\Delta_{\Gr(4,8)}$ we used expressions of fundamental invariants $f_{i}$ (gotten by taking traces of sums of powers of roots) that were normalized so that the coefficients were integers with no common factors. So once one calibrates the choices for the basic fundamental invariants so that they match these conventions, values for the fundamental invariants can be substituted into the expression \eqref{eq:del48} to obtain the value of the discriminant at that point.

\section{Acknowledgements}
We thank the referees of this work whose suggestions improved this work.
We thank the developers of Maple and Macaulay2, where most of our computations were done and checked.
We thank Nathan Ilten, Doug Leonard, and Steven Sam for helpful conversations.
Auburn University supplied computing resources.
Oeding is grateful to UTBM and IMPAN for the support provided during which most of this work was carried out.

This work was partially supported by the grant 346300 for IMPAN from the Simons Foundation and the matching 2015-2019 Polish MNiSW fund as well as by the French ``Investissements d'Avenir'' programme, project ISITE-BFC (contract ANR-15-IDEX-03).
\newcommand{\arxiv}[1]{}
\bibliographystyle{amsalpha}
\bibliography{/Users/lao0004/Dropbox/BibTeX_bib_files/main_bibfile.bib}

\end{document}